\documentclass[12pt]{article}
\usepackage{epsfig,wrapfig,epic}
\usepackage{amscd}
\usepackage{rotating}
\usepackage{wrapfig}

\usepackage{hyperref}
\usepackage{amsfonts}
\usepackage{mathrsfs}
\usepackage{stmaryrd}
\usepackage{shadow}

\usepackage{dsfont}
\usepackage{graphics}
\usepackage{graphicx}
\usepackage{amssymb}
\usepackage{amsmath}
\usepackage{bbm}
\DeclareMathAlphabet{\mathpzc}{OT1}{pzc}{m}{it}

\newtheorem{example}{Example}

\newcommand{\C}{\mathbbm{C}}
\newcommand{\bdx}{\mathbf{x}}
\newcommand{\bdy}{\mathbf{y}}
\newcommand{\bdz}{\mathbf{z}}
\newcommand{\bdo}{\mathbf{0}}
\newcommand{\h}{{{\mbox{\tiny $\mathsf{H}$}}}}
\newcommand{\cK}{\mathpzc{Kernel}}
\newcommand{\cR}{\mathpzc{Range}}
\newcommand{\eig}{\mathpzc{eig}}
\newcommand{\la}{\lambda}
\newcommand{\sg}{\sigma}
\newcommand{\al}{\alpha}
\newcommand{\bdu}{\mathbf{u}}
\newcommand{\bdg}{\mathbf{g}}
\newcommand{\pd}[2]{\frac{\partial #1}{\partial #2}}
\newcommand{\cC}{{\cal C}}
\newcommand{\cE}{{\cal E}}
\newcommand{\spn}{\mathpzc{span}}
\newcommand{\eps}{\varepsilon}
\newcommand{\blb}{\big[\,}
\newcommand{\brb}{\, \big]}
\newcommand{\QED}{${~} $ \hfill \raisebox{-0.3ex}{\large $\Box$}}

\newtheorem{lemma}{Lemma}
\newtheorem{theorem}{Theorem}

\title{Sensitivity and Computation of a Defective Eigenvalue
}

\author{Zhonggang Zeng\thanks{Department of
Mathematics, Northeastern Illinois University, Chicago, IL 60625.
(zzeng@neiu.edu).}}

\begin{document}

\input amssym.def
\maketitle

\begin{abstract}
~A defective eigenvalues is well documented to be hypersensitive 
to data perturbations and round-off errors, making it a formidable challenge 
in numerical computation particularly when the matrix is known through
approximate data.
~This paper establishes a finitely bounded sensitivity of a defective 
eigenvalue with respect to perturbations that preserve
the geometric multiplicity and the smallest Jordan block size.
~Based on this perturbation theory, numerical computation of a defective 
eigenvalue is regularized as a well-posed least squares problem so that
it can be accurately carried out using floating point arithmetic 
even if the matrix is perturbed. 
\end{abstract}



\vspace{-4mm}
\section{Introduction}

\vspace{-4mm}
Computing matrix eigenvalues is one of the fundamental problems in theoretical 
and numerical linear algebra.
~Remarkable advancement has been achieved since the advent of 
the Francis QR algorithm in 1960s. 
~However, it is well documented that multiple and defective eigenvalues are 
hypersensitive to both data perturbations and the 
inevitable round-off.
~For an eigenvalue of a matrix ~$A$ ~associated with the largest
Jordan block size ~$l\times l$ ~while ~$A$ ~is perturbed by ~$\Delta A$, ~the
error bound \cite[p. 58]{chatelin-fraysse}\cite{chatelin-86,Lidskii}
on the eigenvalue deviation is proportional to
~$\|\Delta A\|_2^{1/l}$, ~implying that the accuracy of the computed
eigenvalue in number of digits is a fraction ~$\frac{1}{l}$
~of the accuracy of the matrix data.
~As a result, numerical computation of defective eigenvalues 
remains a formidable challenge.

On the other hand, it has been known that a defective eigenvalue 
disperses into a cluster when the matrix is under 
arbitray perturbations but the mean of the cluster is not hypersensitive
\cite{kato66,ruhe70b}.
~In his seminal technical report\cite{Kahan72}, Kahan proved that the 
sensitivity of an \,$m$-fold eigenvalue is actually bounded by 
~$\frac{1}{m}\|P\|_2$ ~where ~$P$ ~is the spectral projector associated with
the eigenvalue as long as the perturbation is constrained to preserve the
algebraic multiplicity.
~The same proof and the same sensitivity also apply to the mean of the 
eigenvalue cluster emanating from the \,$m$-fold eigenvalue with respect to
perturbations.
~Indeed, using cluster means as approximations to defective eigenvalues
has been extensively applied to numerical computation of Jordan Canonical Forms
and staircase forms, provided that the clusters can be sorted out from the
spectrum.
~This approach includes works of Ruhe \cite{ruhe-70-bit}, 
Sdridhar and Jordan \cite{sri-jor},
and culminated in Golub and Wilkinson's review
\cite{golub-wilkinson} as well as K{\aa}gstr\"{o}m and Ruhe's
{\sc JNF} \cite{kagstrom-ruhe-jnf,kagstrom-ruhe}.
~Theoretical issues have been analyzed in, e.g.
works of Demmel \cite{Dem83,demmel-86} and Wilkinson \cite{wilk-84,wilk-86}.
~Perturbations on eigenvalue clusters are also studied as pseudospectra 
of matrices in works of Trefethon and Embree \cite{tref-emb} as well
as Rump \cite{Rump01,Rump06}.

In this paper we elaborate a different measurement of the sensitivity of a 
defective eigenvalue with respect to perturbations constrained to 
preserve the geometric multiplicity and the smallest Jordan block size. 
~We prove that such sensitivity is also finitely bounded even if the 
multiplicity is not preserved, and it is large only if either the geometric 
multiplicity or the smallest Jordan block size can be increased by a small 
perturbation on the matrix.
~This sensitivity can be small even if the spectral projector norm is large, 
or vice versa.

In computation, perturbations are expected to be arbitrary without preserving 
either the multiplicity or what we refer to as the multiplicity support. 
~We prove that a certain type of pseudo-eigenvalue uniquely exists, is 
Lipschitz 
continuous, is backward accurate and approximates the defective eigenvalue 
with forward accuracy in the same order of the data accuracy,
~making it a well-posed problem for computing a defective eigenvalue via 
solving a least squares problem.
~Based on this analysis, we develop an iterative algorithm {\sc PseudoEig}%
\footnote{A permanent website
~{\tt homepages.neiu.edu/$\sim$zzeng/pseudoeig.html} ~is set up 
to provide Matlab source codes and other resources for 
Algorithm {\sc PseudoEig}.}
that is capable of accurate computation of defective eigenvalues using 
floating point arithmetic
from empirical matrix data even if the spectral projector norm is large and
thus the cluster mean is inaccurate.

\vspace{-4mm}
\section{Notation}

\vspace{-4mm}
The space of dimension ~$n$ ~vectors is ~$\C^n$ ~and the space of
~$m\times n$ ~matrices is ~$\C^{m\times n}$.
~Matrices are denoted by upper case letters ~$A$, ~$X$, ~and ~$G$, ~etc, with 
~$O$ ~representing a zero matrix whose dimensions can be derived from the 
context.
~Boldface lower case letters such as ~$\bdx$ ~and ~$\bdy$ ~represent vectors.
~Particularly, the zero vector in ~$\C^n$ ~is denoted by
~$\bdo_n$ ~or simply ~$\bdo$ ~if the dimension is clear.
~The conjugate transpose of a matrix or vector ~$(\cdot)$ ~is denoted by
~$(\cdot)^\h$, ~and the Moore-Penrose inverse of a matrix ~$(\cdot)$ 
~is ~$(\cdot)^\dagger$.
~The submatrix formed by entries in rows ~$i_1,\ldots,i_2$ ~and columns
~$j_1,\ldots,j_2$ ~of a matrix ~$A$ ~is denoted by ~$A_{i_1:i_2,j_1:j_2}$.
~The kernel and range of a matrix ~$(\cdot)$ ~are denoted by ~$\cK(\cdot)$
~and ~$\cR(\cdot)$ ~respectively.
~The notation ~$\eig(\cdot)$ ~represents the spectrum of a matrix ~$(\cdot)$.

We also consider vectors in product spaces such as ~$\C\times\C^{m\times k}$.
~In such cases, the vector 2-norm is the square root of the sum of squares of 
all components.
~For instance, a vector ~$(\la,X)\in\C\times\C^{n\times k}$ ~can be arranged 
as a column vector ~$\bdu$ ~in ~$\C^{n\,k+1}$ ~and 
~$\|(\la,X)\|_2 = \|\bdu\|_2$ ~regardless of the ordering.
~A zero vector in such a vector space is also denoted by ~$\bdo$.

Let ~$\la_*$ ~be an eigenvalue of a matrix ~$A\in\C^{n\times n}$.
~Its algebraic multiplicity can be partitioned into a non-increasing sequence 
~$\{l_1,\, l_2,\,\dots\}$ 
~of integers called the {\em Segre characteristic}~\cite{dem-edel} 
that are the sizes of elementary Jordan blocks,
and there is a matrix ~$X_*\in\C^{n\times m}$ ~such that
\[ A\,X_* ~=~ X_*\,\left[\begin{array}{ccc}
J_{l_1}(\la_*) & & \\ & J_{l_2}(\la_*) & \\ & & \ddots 
\end{array}\right] ~~\mbox{where}~~
 J_k(\la_*) ~=~ \mbox{\footnotesize $\left[\begin{array}{cccc}
\la_* & 1 & & \\ & \la_* & \ddots & \\ & & \ddots & 1 \\ & & & \la_*
\end{array}\right]_{k\times k}$}.
\]
For convenience, a Segre characteristic is infinite in formality and
the number of nonzero entries is the geometric multiplicity. 
~The last nonzero component of a Segre characteristic, namely the size of the
smallest Jordan block associated with ~$\la_*$, ~is of particular 
importance in our analysis and we shall call it the {\em Segre
characteristic anchor} ~or simply {\em Segre anchor}.

\begin{wrapfigure}{r}{3.3in}
\vspace{-10mm}
\begin{center}
\epsfig{figure=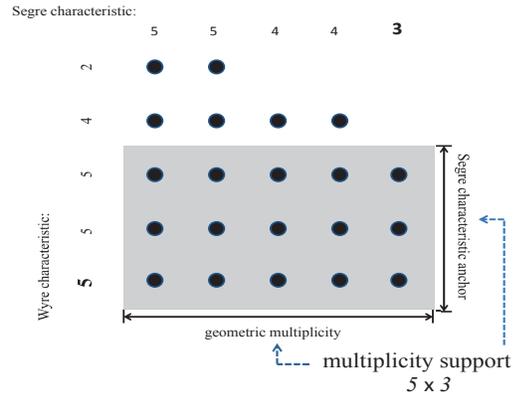,width=3.2in,height=2.1in}
\end{center} \vspace{-6mm}
\caption{Illustration of a ~$5\times 3$ ~multiplicity support.}
\label{f:eigdim} 
\vspace{-7mm}
\end{wrapfigure}
For instance, if ~$\la_*$ ~is an eigenvalue of ~$A$ ~associated with
elementary Jordan blocks 
~$J_5(\la_*),\, J_5(\la_*),\, J_4(\la_*),\, J_4(\la_*)$ ~and
~$J_3(\la_*)$, ~its Segre characteristic is ~$\{5,5,4,4,3,0,\ldots\}$ ~with
a Segre anchor ~$3$. 
~The geometric multiplicity is ~$5$.
~A Segre characteristic along with its conjugate that is called the
Weyr characteristic can be illustrated by a Ferrer's diagram~\cite{dem-edel}
in Fig.~\ref{f:eigdim}, where the geometric multiplicity and the Segre 
anchor represent the dimensions of the base rectangle occupied by 
the equal leading entries of the Weyr characteristic.

For a matrix ~$A$, ~we shall say the {\em multiplicity support}\,
of its eigenvalue ~$\la_*$ ~is ~$m\times k$ ~if the geometric multiplicity of 
~$\la_*$ ~is ~$m$ ~and the Segre anchor is ~$k$. 
%
%
~In this case, there is a unique ~$X_*\in\C^{n\times k}$ ~satisfying the
equations
\begin{eqnarray*}   (A-\la_* I)\,X_* & ~~=~~& X_*\,J_k(0) \\
                             C^\h\,X_* & = & T
\end{eqnarray*}
with proper choices of matrix parameters ~$C\in\C^{n\times m}$ ~and 
\begin{equation}\label{jbitT} T ~~=~~
\left[\begin{array}{ll} 1 & \bdo_{k-1}^\top \\ 
\bdo_{m-1} & O_{(m-1)\times (k-1)} \end{array}\right] ~~\in~~\C^{m\times k}.
\end{equation}
as we shall prove in Lemma~\ref{l:jbit}.
~Here ~$J_k(0)$ ~is a nilpotent upper-triangular matrix of rank ~$k-1$ ~and
can be replaced with any matrix of such kind.
~For integers ~$m,k\le n$, ~we define a holomorphic mapping 
\begin{equation}\label{jbitf}
\begin{array}{ccrcl}
\bdg & ~~:~~ &  \C^{n\times n}\times \C\times\C^{n\times k} & 
~~\longrightarrow~~ & \C^{n\times k}\times \C^{m\times k} \\
& & (G,\la,X) & \longmapsto & \left( \begin{array}{c} (G-\la I)\,X - X\,S \\
C^\h\, X - T  \end{array} \right)
\end{array} 
\end{equation}
that depends on parameters ~$C\in\C^{n\times m}$ ~and
an upper-triangular nilpotent matrix 
\begin{equation}\label{S}
 S ~~=~~ \mbox{\small $\left[\begin{array}{cccc} 0 & s_{12} & \cdots & s_{1k} \\
\vdots & \ddots & \ddots & \vdots \\
\vdots & & \ddots & s_{k-1,k} \\ 0 & \cdots & \cdots & 0 
\end{array}\right]$} ~~~~\mbox{with}~~~~
s_{12}s_{23}\cdots s_{k-1,k} \ne 0
\end{equation}
of rank ~$k-1$. 
~We shall denote the Jacobian and partial Jacobian
\begin{eqnarray*}  
\bdg_{_{G \la X}}(G_0,\la_0,X_0) &~~=~~ &
\left.\pd{\bdg(G,\la,X)}{(G,\la,X)}\right|_{(G,\la,X) = (G_0,\la_0,X_0)} \\
\bdg_{_{\la X}}(G_0,\la_0,X_0) &~~=~~ &
\left.\pd{\bdg(G_0,\la,X)}{(\la,X)}\right|_{(\la,X) = (\la_0,X_0)} 
\end{eqnarray*}
at particular ~$G_0$, ~$\la_0$ ~and ~$X_0$
~that can be considered linear transformations
\begin{equation}\label{gGlaX}
\begin{array}{rcl}
\bdg_{_{G \la X}}(G_0,\la_0,X_0) ~~: ~~~~~~& & \\  
\C^{n\times n}\times \C\times\C^{n\times k} & 
\longrightarrow & \C^{n\times k}\times \C^{m\times k} \\
(G,\la,X) & \longmapsto & \left( \begin{array}{c} 
(G-\la I)\,X_0 +(G_0-\la_0 I)\,X- X\,S 
\\
C^\h\, X  \end{array} \right) 
\end{array} 
\end{equation}
and 
\begin{equation}\label{glaX}
\begin{array}{ccrcl}
\bdg_{_{\la X}}(G_0,\la_0,X_0) & : &  
\C\times\C^{n\times k} & 
\longrightarrow & \C^{n\times k}\times \C^{m\times k} \\
& & (\la,X) & \longmapsto & \left( \begin{array}{c} 
-\la\,X_0 +(G_0-\la_0 I)\,X- X\,S \\
C^\h\, X  \end{array} \right)
\end{array} 
\end{equation}
respectively.
~The actual matrices representing the Jacobians depend on the ordering of 
the bases for the domains and codomains of those linear transformations.
~The Moore-Penrose inverse of a linear transformation such as 
~$\bdg_{_{\la X}}(G_0,\la_0,X_0)^\dagger$ ~is the linear transformation 
whose matrix representation is the Moore-Penrose inverse matrix
of the matrix representation for ~$\bdg_{_{\la X}}(G_0,\la_0,X_0)$ 
~corresponding to the same bases.

\vspace{-4mm}
\section{Properties of the multiplicity support}

\vspace{-4mm}
The following lemma asserts a basic property of the multiplicity support.

\begin{lemma}\label{l:ms}
	~Let ~$A\in\C^{n\times n}$ ~with ~$\la_*\in\eig(A)$
~of multiplicity support ~$m\times k$. 
~Then
\begin{equation}\label{kjran}
\cK\big((A-\la_* I)^{j} \big) ~\subset~\cR(A-\la_* I)
~~~~\mbox{for}~~~~ j = 1, 2, \ldots, k-1.
\end{equation}
Furthermore, there is an open and dense subset ~$\cC$ ~of ~$\C^{n\times m}$
~such that, for every ~$C\,\in\,\cC$,
~the solution ~$\bdx_*$ ~of the equation
\begin{equation}\label{chx10} 
 C^\h\,\bdx ~~=~~ \mbox{\scriptsize 
$\left[\begin{array}{c} 1 \\ \bdo \end{array}\right]$}
~~~~\mbox{for}~~~ \bdx \,\in\, \cK(A-\la_* I)
\end{equation}
uniquely exists and satisfies 
~$\bdx_*  \,\in\,
\mbox{$\Big(\bigcap_{j=1}^{k-1} \cR\big((A-\la_* I)^j\big) \Big)$}
~\setminus~\cR\big((A-\la_* I)^{k}\big)$.
\end{lemma}

\begin{proof}
From the multiplicity support of ~$\la_*$,
~there are ~$m$ ~Jordan blocks of sizes
~$\ell_1 \ge \cdots \ge \ell_m$ ~respectively with ~$\ell_m=k$
~along with ~$m$ ~sequences of generalized eigenvectors 
~$\big\{\bdx_1^{(i)},\bdx_2^{(i)},\ldots,\bdx_{\ell_i}^{(i)}\big\}_{i=1}^m$
~such that ~$(A-\la_* I)\bdx_{j+1}^{(i)} = \bdx_j^{(i)}$ 
~for ~$i=1,\ldots,m$ ~and ~$j=1,\ldots,\ell_i-1$.
~Moreover, ~$\cK\big((A-\la_* I)^j\big)$ ~is spanned by ~$\bdx_l^{(i)}$
~for ~$1\le l\le j$ ~and ~$1\le i\le m$.
~Thus (\ref{kjran}) holds.
~Furthermore
~$(A-\la_* I)^j\,\bdx_{j+1}^{(i)} \,=\,\bdx_1^{(i)}$ 
~for ~$j=1,\ldots,\ell_i-1$ ~and ~$i=1,\ldots,m$. 
~Namely every ~$\bdz\in\cK(A-\la_* I)$ ~is in 
~$\bigcap_{j=1}^{k-1} \cR\big((A-\la_* I)^j\big)$ ~since ~$\ell_i\ge k$
~for all ~$i$.
~However, ~$\bdx_1^{(m)} \not\in \cR\big((A-\la_* I)^k\big)$ ~since 
\[ (A-\la_* I)^k\, \big[\bdx_1^{(m)},\cdots,\bdx_{\ell_m}^{(m)}\big] ~~=~~ 
\big[\bdx_1^{(m)},\cdots,\bdx_{\ell_m}^{(m)}\big]\, J_k(0)^k ~~=~~ O
\]
and ~$\C^n$ ~is the direct sum of those invariant subspaces,
~implying at least one vector in the basis of
~$\cK(A-\la_* I)$ ~is not in ~$\cR\big((A-\la_* I)^k\big)$ ~so
the dimension of the subspace 
~${\cal K} = \cK(A-\la_* I)\cap\cR\big((A-\la_* I)^k\big)$ 
~is less than ~$m$. 

Let columns of ~$N\,\in\,\C^{n\times m}$ ~form an orthonormal basis for
~$\cK(A-\la_* I)$ ~and denote ~$\cC_0 \,=\,\big\{ C\in\C^{n\times m} ~\big|~
(C^\h N)^{-1} ~\mbox{exists}\big\}$,
~which is open since ~$(C+\Delta C)^\h N$ ~is invertible
if ~$C^\h N$ ~is invertible and ~$\|\Delta C\|_2$ ~is sufficiently small.
~For any ~$C\,\not\in\,\cC_0$ ~so that ~$C^\h N$ ~is rank-deficient, 
~we have ~$(C-\eps\,N)^\h N \,=\, C^\h N-\eps\,I$ ~is invertible for
all ~$\eps\,\not\in\,\eig(C^\h N)$ ~so ~$C-\eps\,N\,\in\,\cC_0$ ~and
~$\cC_0$ ~is thus dense.
~For every ~$C\,\in\,\cC_0$ 
~the equation (\ref{chx10}) then has a unique solution 
\[ \bdx_* ~~=~~
N\,(C^\h\,N)^{-1}\,\mbox{\scriptsize 
$\left[\begin{array}{c} 1\\ \bdo\end{array}\right]$}
\]
Let ~$\cC\,\subset\,\cC_0$ ~such that
the ~$\bdx_*\,\not\in\,{\cal K}$ ~for every ~$C\,\in\,\cC$.
~Clearly ~$\cC$ ~is open since, for every ~$C\,\in\,\cC$, ~we have
~$\hat\bdx\,=\,
N\,\big((C+\Delta C)^\h\,N\big)^{-1}\,\mbox{\scriptsize
$\left[\begin{array}{c} 1\\ \bdo\end{array}\right]$}\,\not\in\,{\cal K}$
~for small ~$\|\Delta C\|_2$ ~and thus ~$C+\Delta C\,\in\,\cC$.
~To show ~$\cC$ ~is dense in ~$\cC_0$, ~let ~$C\,\in\,\cC_0$ ~with
the corresponding ~$\bdx_*\,\in\,{\cal K}$.
~Since ~$\dim({\cal K})<m$, ~there is a unit vector 
~$\hat\bdx\,\in\,\cK(A-\la_* I)\,\setminus\,{\cal K}$.
~For any ~$\eps\ge 0$, ~let
~$D^{(\eps)} \,=\,
-\frac{\bdx_*+\eps\,\hat\bdx}{\|\bdx_*+\eps\,\hat\bdx\|_2^2}
\,\hat\bdx^\h\,C$.
~There is a ~$\mu>0$ ~such that
~$\big\|D^{(\eps)}\big\|_2 \,\le\, \mu$ ~for all ~$\eps\,\in\,[0,1]$ 
~since ~$\min_{\eps\in [0,1]}\|\bdx_*+\eps\,\hat\bdx\|_2 > 0$.
~Then 
\begin{equation*} \big(C+\eps\,D^{(\eps)}\big)^\h (\bdx_*+\eps\,\hat\bdx) \,=\,
C^\h\,\bdx_* + \eps\,C^\h\hat\bdx - \eps\,C^\h\,\hat\bdx\,
\frac{(\bdx_*+\eps\,\hat\bdx)^\h}{\|\bdx_*+\eps\,\hat\bdx\|_2^2}
(\bdx_*+\eps\,\hat\bdx) \,=\,
\mbox{\footnotesize $\left[\begin{array}{c} 1 \\ \bdo \end{array}\right]$}
\end{equation*}
with 
~$\bdx_*+\eps\,\hat\bdx \,\in\,\cK(A+\la_* I)\,\setminus {\cal K}$ 
~for all ~$\eps\,\in\,(0,1)$ ~and 
~$\big\|\eps\,D^{(\eps)}\big\|_2<\eps \,\mu$.
~Namely, ~$C+\eps\,D^{(\eps)} \in \cC$ ~for sufficiently small 
~$\eps$ ~and approaches to ~$C$
~when ~$\eps\,\rightarrow\,0$, ~implying ~$\cC$ ~is dense in ~$\cC_0$
~that is dense in ~$\C^{n\times m}$ ~so the lemma is proved.
~\QED
\end{proof}

The following lemma sets the foundation for our sensitivity analysis
and algorithm design on a defective eigenvalue 
by laying out critical properties of the mapping
(\ref{jbitf}).

\begin{lemma}\label{l:jbit}
~Let ~$A\in\C^{n\times n}$ ~with ~$\la_*\,\in\,\eig(A)$
~of multiplicity support ~$m_*\times k_*$ ~and ~$\bdg$ ~be as
in {\em (\ref{jbitf})} with ~$S$ ~and ~$T$ ~as in {\em (\ref{S})} 
~and {\em (\ref{jbitT})} respectively.
~The following assertions hold. 

\vspace{-4mm}
\begin{itemize}\parskip-0.5mm
\item[\em (i)] 
For almost all ~$C\in\C^{n\times m}$ ~as a parameter for ~$\bdg$, 
~an ~$X_* \in\, \C^{n\times k}$ 
~exists such that ~$\bdg(A,\la_*,X_*) = \bdo$ ~if and only if 
~$m\le m_*$ ~and ~$k\le k_*$.
~Such an ~$X_*$ ~is unique if and only if ~$m=m_*$.
\item[\em (ii)] ~Let ~$m\le m_*$ ~and ~$k\le k_*$.
~For almost all ~$C\in\C^{n\times m}$ ~in ~$\bdg$ ~with
~$\bdg(A,\la_*,X_*) = \bdo$,
~the linear transformation ~$\bdg_{_{G \la X}}(A,\la_*,X_*)$ ~is surjective,
~and ~$\bdg_{_{\la X}}(A,\la_*,X_*)$ ~is injective
if and only if ~$m=m_*$ ~and ~$k=k_*$.
\item[\em (iii)]
~~Let ~$m=m_*$, ~$k=k_*$  ~and ~$\bdg(A,\la_*,X_*) = \bdo$.
~Then ~$C$ ~and ~$S$ ~can be modified so that the columns of 
~$X_*$ ~are orthonormal.
\end{itemize}
\end{lemma}

\begin{proof}
Let ~$N\in\C^{n\times m_*}$ ~be a matrix whose columns span 
~$\cK(A-\la_* I)$.
~We shall prove the assertion (i) by an induction.
~For almost all ~$C\in\C^{n\times m}$, ~the matrix ~$C^\h N$ ~is 
of full row rank if ~$m\le m_*$ ~so that there is a ~$\bdu\,\in\,\C^{m_*}$ 
~such that ~$(C^\h\,N)\, \bdu \,=\, T_{1:m,1}$
~while ~$\bdu$ ~is unique if and only if ~$m=m_*$.
~For ~$m\le m_*$, ~let ~$\bdx_1=N\,\bdu$ ~and assume 
~$\bdx_1,\ldots,\bdx_j \in \C^{n}$ 
~are obtained such that ~$1\le j < k$ ~and
\begin{eqnarray*} (A-\la_* I) \big[\bdx_1,\cdots,\bdx_j\big] &~~=~~ &
\big[\bdx_1,\cdots,\bdx_j\big]\, S_{1:j,1:j} \\
C^\h \big[\bdx_1,\cdots,\bdx_j\big] & = & T_{1:m,1:j}
\end{eqnarray*}
Then ~$\bdx_1,\ldots,\bdx_j \in \cK\big((A-\la_* I)^j\big)$ ~from
~$(S_{1:j,1:j})^j = O$, ~and (\ref{kjran}) implies that
\[ (A-\la_* I)\, \bdx ~~=~~ s_{1,j+1}\bdx_1 + \cdots + s_{j,j+1}\bdx_j
~~\equiv~~ [\bdx_1,\ldots,\bdx_j]\,S_{1:j,j+1}
\]
has a particular solution ~$\bdu\,\in\,\C^n$
~and a unique solution 
~$\bdx_{j+1} \,=\, \bdu - N\,(C^\h N)^{-1} C^\h \bdu$
~such that ~$C^\h\,\bdx_{j+1} = \bdo$ ~when ~$m=m_*$.
~By induction, there is a matrix
~$X_*= \big[\bdx_1,\cdots,\bdx_k\big]\in\C^{n\times k}$ ~such that 
~$(\la_*,X_*)$ ~is a solution to the system
~$\bdg(A,\la,X) = \bdo$ ~and ~$X_*$ ~is unique if and only if ~$m=m_*$.
~The assertion (i) is proved.

Assume ~$\bdg(A,\la_*,X_*) = \bdo$ ~and write 
~$X_* = \blb \bdx_1,\cdots,\bdx_k\brb$.
~Then ~$\bdx_1\ne\bdo$ ~and 
~$\bdx_j\,\in\,\cK\big((A-\la_* I)^j\big)\setminus
\cK\big((A-\la_* I)^{j-1}\big)$
~for ~$j=1,\ldots,k$ 
~by ~$S_{1:j,1:j}$ ~being upper triangular nilpotent.
~So ~$X_*$ ~is of full column rank and ~$X_*^\dagger X_*=I$.
~Furthermore, 
the Jacobian ~$\bdg_{_{G \la X}}(A,\la_*,X_*)$ ~is surjective since,
for any ~$U\,\in\,\C^{n\times k}$ ~and ~$V\,\in\,\C^{m\times k}$, 
~a straightforward calculation using (\ref{gGlaX}) yields
\[  \bdg_{_{G \la X}}(A,\la_*,X_*)\big((U-(A-\la_*I)\,C^{\h\dagger}V+
C^{\h\dagger}V S)\,X_*^\dagger,\,0,\, C^{\h \dagger}\, V\big) ~~=~~
\mbox{\scriptsize $\left(\begin{array}{c} U \\ V
\end{array}\right)$} 
\]
using ~$C^\h\,C^{\h \dagger} = I$ ~when ~$C$ ~is of full column rank.
~Let ~$(A,\la_*,\hat{X})$ ~be a zero of ~$\bdg$ ~and assume ~$m=m_*$ ~and
~$k=k_*$.
~Then, for almost all ~$C\in\C^{n\times m}$, ~the solution ~$\bdu = \hat\bdu$ 
~of the equation ~$(C^\h\,N)\,\bdu = T_{1:m,1}$ ~is unique and
the first column of ~$\hat{X}$, ~from Lemma~\ref{l:ms}, ~is 
\begin{equation}\label{bdx1in}
 \hat\bdx_1 ~~=~~ N\,\hat\bdu  ~~\in~~ 
\left(\mbox{$\bigcap_{j=1}^{k-1} \cR\big((A-\la_* I)^j\big)$} \right)
~\setminus~\cR\big((A-\la_* I)^{k}\big).
\end{equation}
Assume, ~for a ~$(\sg,Y)\in\C\times\C^{n\times k}$, 
~its image 
~$\bdg_{_{\la X}}(A,\la_*,\hat{X}) \big(\sg, Y\big) \,=\,\bdo$.
~By (\ref{glaX}),
\begin{eqnarray}
-\sg\,\hat{X} + (A-\la_* I)\,Y - Y\,S &~~=~~& O \label{sxaly1} \\
C^\h\,Y & ~~=~~ &  O. \label{sxaly2}
\end{eqnarray}
Right-multiplying both sides of the equation (\ref{sxaly1}) by ~$S$ ~yields
\begin{align*}   
Y\,S^2 + \sg\,\hat{X}\,S &~~=~~  (A-\la_* I)\,Y\,S \nonumber \\
  &~~=~~ (A-\la_* I)^2\,Y - \sg\,(A-\la_* I)\,\hat{X} & 
\mbox{(by (\ref{sxaly1}))}
\\ 
&~~=~~  (A-\la_* I)^2\,Y - \sg\,\hat{X}\,S,
& \mbox{(by ~$(A-\la_* I)\hat{X}=\hat{X}S$)}
\end{align*}
namely 
\[(A-\la_* I)^2\,Y \,=\,Y\,S^2 + 2\,\sg\,\hat{X}\,S.
\]
Continuing the process of recursive right-multiplying the equation by 
~$S$ ~leads to
\begin{eqnarray*}  (A-\la_* I)^k\,Y & ~~=~~& 
Y\,S^k + k\,\sg\,\hat{X}\,S^{k-1}
~~=~~ \,k\,\sg\,s_{12}\,s_{23}\,\cdots\,s_{k-1,k} 
\blb O_{n\times (k-1)}, ~\hat\bdx_1\brb
\end{eqnarray*}
with ~$s_{12}\,s_{23}\,\cdots\,s_{k-1,k} \ne 0$.
~Hence ~$\sg=0$ ~due to (\ref{bdx1in}).
~Denote columns of ~$Y$ ~as ~$\bdy_1,\ldots,\bdy_k\in\C^{n}$. 
~Then the first columns of the equations (\ref{sxaly1}) and (\ref{sxaly2}) are
~$(A-\la_* I)\,\bdy_1 \,=\, \bdo$ ~and ~$C^\h\,\bdy_1 \,=\, 0$
~that imply ~$\bdy_1=\bdo$.
~For ~$1\le j< k$, using ~$\sg=0$ ~and ~$\bdy_1=\cdots=\bdy_j=\bdo$ ~on the 
~$(j+1)$-th columns of the equations (\ref{sxaly1}) and (\ref{sxaly2}) we have
~$\bdy_{j+1} = \bdo$.
~Thus ~$Y=O$.
~As a result, ~$(A,\la_*,\hat{X})$ ~is a zero of ~$\bdg$ ~with injective
partial Jacobian ~$\bdg_{_{\la X}}(A,\la_*,\hat{X})$. 

If ~$m<m_*$, ~the solution ~$(\la_*,X_*)$ ~of ~$\bdg(A,\la,X)=\bdo$ ~is 
on an algebraic variety of a positive dimension and thus 
~$\bdg_{_{\la X}}(A,\la_*,X_*)$ ~is not injective.
~Let ~$m=m_*$, ~we now prove the partial Jacobian 
~$\bdg_{_{\la X}}(A,\la_*,\hat{X})$ ~is injective {\em only if} ~$k=k_*$.
~Assume ~$k<k_*$ ~and ~write ~$\hat{X} = \blb \hat\bdx_1,\cdots,\hat\bdx_k\brb$.
~Since ~$\bdg(A,\la_*,\hat{X})=\bdo$ ~and ~$S$ ~is upper-triangular nilpotent,
hence ~$\hat\bdx_j \,\in\,\cK\big((A-\la_* I)^j\big)$ ~for ~$j=1,\ldots,k$.
~Then ~$k<k_*$ ~implies ~$\hat\bdx_1,\ldots,\hat\bdx_k \,\in\,\cR(A-\la_*I)$.
~For almost all ~$C\in\C^{n\times m}$, ~the matrix
~{\scriptsize $\left[\begin{array}{c} A-\la_* I \\ C^\h \end{array}\right]$}
~is of full column rank and the vector
~$\bdy_1 \,=\, \frac{1}{s_{12}}\,\hat\bdx_2$ ~is the unique solution
to the linear system ~$\mbox{\scriptsize 
$\left[\begin{array}{c} A-\la_* I \\ C^\h \end{array}\right]$}\,\bdz 
\,=\, \mbox{\scriptsize
$\left[\begin{array}{c} \hat\bdx_1 \\ \bdo \end{array}\right]$}$.
~Using an induction, assume 
~$\bdy_1,\ldots,\bdy_j \in \spn\{\hat\bdx_2,\ldots, \hat\bdx_{j+1}\}$ 
~for any ~$j<k$ ~such that
\begin{eqnarray*}
-\blb \hat\bdx_1,\cdots,\hat\bdx_j\brb 
+ (A-\la_* I)\blb \bdy_1,\cdots,\bdy_j\brb
- \blb \bdy_1,\cdots,\bdy_j\brb\,S_{1:j,1:j} & ~~=~~& O \\
C^\h\,\blb \bdy_1,\cdots,\bdy_j\brb & ~~=~~& O .
\end{eqnarray*}
Then ~$\bdy_1,\ldots,\bdy_j \,\in\,\cK\big((A-\la_* I)^{j+1}\big)$ ~and
(\ref{kjran}) imply that there is a unique vector
~$\bdz=\bdy_{j+1}\in\spn\{\bdx_1,\ldots,\hat\bdx_{j+1}\}$ ~satisfying
\[ (A-\la_* I) \bdz ~~=~~ 
\hat\bdx_{j+1} + s_{1,j+1}\bdy_1+\cdots+s_{j,j+1}\bdy_j ~~~~\mbox{and}~~~
C^\h\,\bdz ~~=~~ \bdo.
\]
Write ~$Y = \blb \bdy_1,\cdots,\bdy_k \brb$.
~We have ~$\bdg_{_{\la X}}(A,\la_*,\hat{X})\big(1, Y\big) \,=\, \bdo$
~and thus the partial Jacobian ~$\bdg_{_{\la X}}(A,\la_*,\hat{X})$ ~is not
injective.
~As a result, the assertion (ii) is proved.

We now prove (iii). 
~Let ~$\bdg(A,\la_*,\hat{X})\,=\,\bdo$ ~for certain parameters ~$C$ ~and
~$S$.
~We can assume ~$C$ ~and $S$ ~are properly scaled so that 
~$\|\hat{X}_{1:n,1}\|_2=1$.
~Reset columns ~$C_{1:n,1}$ ~as ~$\hat{X}_{1:n,1}$, 
~$\hat{X}_{1:n,2:k}$ ~as ~$\hat{X}_{1:n,2:k} - \hat{X}_{1:n,1}\, 
(\hat{X}_{1:n,1})^\h \hat{X}_{1:n,2:k}$
~and ~$S_{1,1:k}$ ~as ~$S_{1,1:k}+(\hat{X}_{1:n,1})^\h \hat{X}_{1:n,2:k}
\,S_{2:k,1:k}$ ~so that ~$\bdg(A,\la_*,\hat{X})=\bdo$ ~still holds
and ~$(\hat{X}_{1:n,2:k})^\h \hat{X}_{1:n,1} =\bdo$.
~As a result, there is a thin QR decomposition ~$\hat{X} = Q\,R$ ~with 
~$R_{1,1:k}=[1,0,\cdots,0]$.
~Reset ~$X_*=Q$ ~and ~$S$ ~as ~$R\,S\,R^{-1}$.
~It is thus a straightforward verification that ~$\bdg(A,\la_*,X_*)=\bdo$
~with ~$(X_*)^\h X_* = I$.
~\QED
\end{proof}

\vspace{-4mm}
\section{Sensitivity of a defective eigenvalue}

\vspace{-4mm}
Based on Lemma~\ref{l:jbit} and the Implicit 
Function Theorem, the following lemma establishes 
the defective eigenvalue as a holomorphic function
of certain entries of the matrix.

\begin{lemma}\label{l:jbith}
~Assume ~$A\in\C^{n\times n}$ ~and ~$\la_*\in\eig(A)$ ~of multiplicity
support ~$m\times k$.
~Let ~$\bdg$ ~be defined in {\em (\ref{jbitf})} using proper parameters
~$C\in\C^{n\times m}$ ~and ~$S\in\C^{n\times k}$ ~so that 
~$\bdg(A,\la_*,X_*) = \bdo$ 
~with a surjective
~$\bdg_{_{G \la X}}(A,\la_*,X_*)$ ~and an injective 
~$\bdg_{_{\la X}}(A,\la_*,X_*)$.
~There is a neighborhood ~$\Omega$ ~of certain ~$\bdz_*$ ~in 
~$\C^{n^2-m\,k+1}$ ~and a neighborhood ~$\Sigma$ ~of ~$(A,\la_*,X_*)$
~in ~$\C^{n\times n}\times\C\times\C^{n\times k}$ ~along with
holomorphic mappings 
~$G\,:\,\Omega\longrightarrow\C^{n\times n}$, 
~$\la\,:\,\Omega\longrightarrow\C$ ~and
~$X\,:\,\Omega\longrightarrow\C^{n\times k}$ ~with
~$\big(G(\bdz_*),\la(\bdz_*),X(\bdz_*)\big) = (A,\la_*,X_*)$ ~such that
~$\bdg(G_0,\la_0,X_0) = \bdo$ ~at any point ~$(G_0,\la_0,X_0)\in\Sigma$ 
~if and only if there is a ~$\bdz_0\in\Omega$ ~such that
~$(G_0,\la_0,X_0) = (G(\bdz_0),\la(\bdz_0),X(\bdz_0))$.

\end{lemma}

\begin{proof}
~Since the mapping ~$(G,\la,X)\,\mapsto\,\bdg(G,\la,X)$ ~has a surjective
Jacobian
~$\bdg_{_{G \la X}}(A,\la_*,X_*)$ ~to 
~$\C^{n\times k}\times\C^{m\times k}$ ~and an injective
~$\bdg_{_{\la X}}(A,\la_*,X_*)$ ~from ~$\C\times\C^{n\times k}$, 
~there are ~$m\,k-1$ ~entries of the variable ~$G\in\C^{n\times n}$ 
~forming a variable ~$\bdy$ ~such that the partial Jacobian 
~$\bdg_{_{\bdy \la X}}(A,\la_*,X_*)$ ~is invertible. 
~By the Implicit Function Theorem, the remaining entries of ~$G$ ~excluding 
~$\bdy$ ~form a variable vector ~$\bdz\in\C^{n^2-m\,k+1}$ ~so
that the assertion holds.
\QED
\end{proof}

From the proof of Lemma~\ref{l:jbith}, the components
of the variable ~$\bdz$ ~are identical to ~$n^2-m\,k+1$ 
entries of the matrix ~$G(\bdz)$.
~We can now establish one of the main theorems of this paper.

\begin{theorem}[Eigenvalue Sensitivity Theorem]\label{t:condmk}
~The sensitivity of an eigenvalue is finitely bounded if
its multiplicity support is preserved.
~More precisely, let the matrix ~$A\in\C^{n\times n}$ ~and 
~$\la_*\in\eig(A)$ ~with 
a multiplicity support ~$m\times k$. 
~There is a neighborhood ~$\Phi$ ~of ~$(A,\la_*)$ ~in 
~$\C^{n\times n}\times\C$ ~and a neighborhood
~$\Omega$ ~of certain ~$\bdz_*$ ~in ~$\C^{n^2-m\,k+1}$ 
~along with holomorphic mappings 
~$G~:~\Omega\rightarrow\C^{n\times n}$ ~and 
~$\la~:~\Omega\rightarrow\C$ 
~with ~$(A,\la_*) =\big(G(\bdz_*),\la(\bdz_*)\big)$ ~such
that every ~$(\tilde{A},\, \tilde\la)\in\Phi$ ~with 
~$\tilde\la\in\eig\big(\tilde{A}\big)$ ~of multiplicity support ~$m\times k$
~is equal to ~$\big(G(\tilde\bdz),\la(\tilde\bdz)\big)$ ~for certain 
~$\tilde\bdz\in\Omega$.
~Furthermore,
\begin{equation}\label{laine}
\limsup_{\bdz\rightarrow \bdz_*}
\frac{\big|\la(\bdz) - \la_*\big|}{\|G(\bdz)-A\|_{_F}} ~~\le~~  \left\|
\bdg_{_{\la X}}(A,\la_*,X_*)^\dagger \right\|_2 
~~<~~  \infty 
\end{equation}
where ~$X_*\in\C^{n\times k}$ ~satisfies ~$\bdg(A,\la_*,X_*)= \bdo$ ~for the 
mapping ~$\bdg$ ~defined in {\em (\ref{jbitf})} ~that renders columns of 
~$X_*$ ~orthonormal.
\end{theorem}

\begin{proof}
~Let ~$\Sigma$ ~and ~$\Omega$ ~be the neighborhoods specified
in Lemma~\ref{l:jbith} along with the holomorphic mappings ~$G$ ~and ~$\la$.
~For any ~$(\tilde{A},\tilde\la)$ ~sufficiently close to ~$(A,\la_*)$
~with ~$\tilde\la\,\in\,\eig(A)$ ~of multiplicity support ~$m\times k$,
~the matrix {\scriptsize ~$\left[\begin{array}{c} \tilde{A}-\tilde\la I \\
C^\h \end{array}\right]$} ~is of full rank so there is a unique ~$\tilde{X}$
~such that ~$\bdg(\tilde{A},\tilde\la,\tilde{X})=\bdo$.
~Furthermore, the linear transformation 
~$X\,\mapsto\,\big( (\tilde{A}-\tilde\la I)\,X-X\,S,~C^\h X\big)$
~is injective from ~$\C^{n\times k}$ ~to ~$\C^{n\times k}\times\C^{m\times k}$,
~implying ~$\|\tilde{X}-X_*\|_{_F}$ ~can be as small as needed so that
~$(\tilde{A},\tilde\la,\tilde{X})\,\in\,\Sigma$ ~and thus
~$(\tilde{A},\tilde\la)\,=\,(G(\bdz),\la(\bdz))$ ~for certain 
~$\bdz\,\in\,\Omega$.
~Consequently, ~the neighborhood ~$\Phi$ ~of ~$(A,\la_*)$ ~exists.

From Lemma~\ref{l:jbith}, we have
~$\bdg\big(G(\bdz),\la(\bdz),X(\bdz)\big) \,\equiv\, \bdo$
~for all ~$\bdz\in\Omega$.
~As a result,
\begin{eqnarray*}
\bdo &~~=~~ & \Big(
\mbox{\small
$\pd{\bdg(G(\bdz),\la(\bdz),X(\bdz))}{\bdz}$}\Big|_{\bdz=\bdz_*} 
\Big)
(\bdz-\bdz_*)\\ 
& ~~=~~&
\bdg_{_G}(A,\la_*,X_*)\,G_\bdz(\bdz_*)\,(\bdz-\bdz_*) +
\bdg_{_{\la X}}(A,\la_*,X_*)\,
\Big( \mbox{\small 
$\pd{(\la(\bdz),X(\bdz))}{\bdz}$}\Big|_{\bdz=\bdz_*}\Big)
\,(\bdz-\bdz_*)
\end{eqnarray*}
implying
\begin{eqnarray*}
\big|\la(\bdz)-\la_*\big|  &~~\le~~&
\big\|\big(\la(\bdz),X(\bdz)\big)-\big(\la_*,X_*\big) \big\|_2 \\ 
& = & \Big\|
\mbox{\small $\pd{(\la(\bdz),X(\bdz))}{\bdz}$}
\Big|_{\bdz=\bdz_*}\,(\bdz-\bdz_*) \Big\|_2 
+ O\big(\big\|\bdz-\bdz_*\big\|_2^2\,\big) \\ 
& = &
\left\| \bdg_{_{\la X}}(A,\la_*,X_*)^\dagger\,\bdg_{_G}(A,\la_*,X_*)\,
G_\bdz(\bdz_*)\,(\bdz-\bdz_*) \right\|_2 
+  O\big(\big\|\bdz-\bdz_*\big\|_2^2\,\big) \\ 
& \le &  \left\|\bdg_{_{\la X}}(A,\la_*,X_*)^\dagger\right\|_2\,
\big\|G(\bdz)-A\big\|_F +  O\big(\big\|\bdz-\bdz_*\big\|_2^2\,\big) 
\end{eqnarray*}
since the partial Jacobian 
~$\bdg_{_G}(A,\la_*,X_*)$ ~is the linear transformation ~$G\,\mapsto\,G\,X_*$
with a unit operator norm due to orthonormal columns of ~$X_*$,
~leading to ~(\ref{laine}).
~The norm ~$ \left\|\bdg_{_{\la X}}(A,\la_*,X_*)^\dagger\right\|_2$ ~is 
finite because ~$\bdg_{_{\la X}}(A,\la_*,X_*)$ ~is injective by 
Lemma~\ref{l:jbit}. 
\QED
\end{proof}

In light of Theorem~\ref{t:condmk}, we introduce the ~$m\times k$ 
~{\em condition number} 
\begin{equation}\label{condmk}
\tau_{A,m\times k}(\la_*) ~~:=~~  
\inf_{C,S} \left\|\bdg_{_{\la X}} (A,\la_*,X_*)^\dagger\right\|_2
\end{equation}
of an eigenvalue ~$\la_*\in\eig(A)$ 
~where ~$\bdg$ ~is as in (\ref{jbitf}) and the infimum is 
taken over all the proper choices of matrix parameters ~$C$ ~and ~$S$ ~that 
render the columns of the unique ~$X_*$ ~orthonormal so that 
~$\bdg(A,\la_*,X_*)=\bdo$.
~We shall refer to ~$\tau_{A,m\times k}(\la_*)$ ~as the 
{\em multiplicity support condition number} ~if the specific ~$m$ ~and ~$k$ 
~are irrelevant in the discussion.
~From Lemma~\ref{l:jbit}, the ~$m\times k$ ~condition number is infinity 
{\em only if}\, either ~$m$ ~is less than the actual geometric multiplicity 
or ~$k$ ~is less than the Segre anchor. 
~Consequently, the condition number 
~$\tau_{A,m\times k}(\la_*)$ ~is large only if ~$A$ ~is close to a matrix 
~$\tilde{A}$ ~that possesses an eigenvalue ~$\tilde\la \approx \la_*$ 
~whose multiplicity support is ~$\tilde{m}\times\tilde{k}$ ~with
either ~$\tilde{m}>m$ ~or ~$\tilde{k}>k$.
~As a special case, ~the condition number
~$\tau_{A,1\times 1}(\la_*)$ ~measures the sensitivity of a simple eigenvalue
~$\la_*$.

We can now revisit the old question: 
\begin{quote}
{\em Is a defective eigenvalue hypersensitive to perturbations?}
\end{quote}
The answer is not as simple as the question may seem to be.
~It is well documented in the literature that, under an {\em arbitrary}\, 
perturbation ~$\Delta A$ ~on the matrix ~$A$, ~a defective eigenvalue of ~$A$
~generically disperses into a cluster of eigenvalues with an error 
bound proportional to ~$\big\|\Delta A\big\|_2^{\frac{1}{l}}$ ~where ~$l$
~is the size of the largest Jordan block associated with the eigenvalue
\cite[p.\,58]{chatelin-fraysse}\cite{chatelin-86,Lidskii}.
~Similar and related sensitivity results can be found in the works such as
\cite{BurOve92,lippert-edelman,MoBuOvr97}.
~This error bound implies that the asymptotic sensitivity of a defective 
eigenvalue is infinity, and only a fraction ~$\frac{1}{l}$ ~of the data
accuracy passes on to the accuracy of the eigenvalue.
~For instance, if the largest Jordan block is ~$5\times 5$, ~only three 
correct digits can be expected from the computed eigenvalues regarding the
defective eigenvalue since one fifth the hardware precision (about 16 digits)
remains in the forward accuracy.

It is also known that the mean of the cluster emanating from the defective
eigenvalue under perturbations is not hypersensitive \cite{kato66,ruhe70b}.
~Kahan is the first to discover the finite sensitivity 
~$\frac{1}{m}\,\big\|P\|_2$ ~of a multiple eigenvalue under constrained 
perturbations that preserve the algebraic multiplicity ~$m$, ~where ~$P$ ~is 
the spectral projector associated with the eigenvalue.
~This spectral projector norm is large only if a small perturbation on the
matrix can increase the multiplicity \cite{Kahan72}.
~As pointed out by Kahan, the seemingly infinite sensitivity of 
a multiple eigenvalue may not be a conceptually meaningful measurement for
the condition of a {\em multiple eigenvalue} since arbitrary perturbations 
do not maintain the characteristics of the eigenvalue as being multiple. 
~Theorem~\ref{t:condmk} sheds light on another intriguing and pleasant property 
of a defective eigenvalue: ~Its algebraic multiplicity does {\em not}\, need 
to be maintained under data perturbations for its sensitivity to be under 
control, as long as the geometric multiplicity {\em and} \,the 
Segre anchor are preserved.
~As a result, the condition number ~$\tau_{A,m\times k}(\la_*)$ ~provides a 
new and different measurement on the sensitivity of a {\em defective} 
eigenvalue ~$\la_*$ ~when its multiplicity support is preserved.

The same eigenvalue can be ill-conditioned in the spectral 
projector norm while being well conditioned in multiplicity support condition
number and vice versa (c.f.  Example~\ref{e:4} in \S\ref{s:w}) 
with no contradiction whatsoever.

More importantly, the finite sensitivity enables accurate numerical computation
of a defective eigenvalue from imposing the constraints on the 
multiplicity support, as we shall demonstrate in later sections.
~Even if perturbations are unconstrained, the problem of computing a defective
eigenvalue may not have to be hypersensitive at all if the problem is properly
generalized, i.e. regularized.
~We shall prove in Theorem~\ref{t:pet} that the ~$m\times k$ ~condition number 
still provides the finitely bounded sensitivity of ~$\la_*$ ~as what we call 
the ~$m\times k$ ~pseudo-eigenvalue of ~$A$, ~and this condition number is 
large only if ~$m$ ~or ~$k$ ~can be increased by small perturbations.

There are further subtleties on the condition of a defective eigenvalue.
~The sensitivity is finitely bounded if the multiplicity or the multiplicity
support of {\em the}\, eigenvalue is preserved.
~Denote the collection of ~$n\times n$ ~complex matrices having an 
eigenvalue that shares the same multiplicity support ~$m\times k$ ~as
~$\cE_{m\times k}^n$.
~Every ~$A\,\in\, \cE_{m\times k}^n$ 
~has an eigenvalue ~$\la_*$ ~along with an  ~$X_*$ ~such that
~$(A,\la_*,X_*)$ ~belongs to an algebraic variety defined by the solution set
of the polynomial system ~$\bdg(G,\la,X)=\bdo$.
~The set ~$\cE_{m\times k}^n$ ~is not a manifold in general
so the Tubular Neighborhood Theorem does not apply.
~As a result, maintaining a multiplicity support 
is not enough to dampen the sensitivity of a particular 
defective eigenvalue with that multiplicity support.
~The matrix staying on ~$\cE_{m\times k}^n$ ~does not guarantee the finite
sensitivity of a defective eigenvalue.
~If a matrix ~$A\in\cE_{m\times k}^n$ ~has two eigenvalues of the same
multiplicity support ~$m\times k$, ~then ~$A$ ~is in the intersection of
images of two holomorphic mappings described in Lemma~\ref{l:jbith}.
~When ~$A$ ~drifts on ~$\cE_{m\times k}^n$, ~the multiplicity support 
~$m\times k$ ~may be maintained for one eigenvalue but lost on the other.
~Consequently, the {\em other} defective eigenvalue still disperses into a 
cluster.

\vspace{-4mm}
\section{A well-posed defective eigenvalue problem}

\vspace{-4mm}
A mathematical problem is said to be well-posed if its solution satisfies 
three crucial properties: ~Existence, uniqueness and Lipschitz continuity.
~The problem of finding an eigenvalue of a matrix in its conventional 
meaning is ill-posed when the eigenvalue is defective because 
the sensitivity of the eigenvalue is infinite with respect to arbitrary 
perturbations on the matrix.
~Lacking Lipschitz continuity with respect to data, 
such a problem is not suitable for numerical computation unless the problem
is properly modified, or better known as being {\em regularized}.

We can alter the problem of 
\vspace{-4mm}
\begin{quote}
\em finding an eigenvalue of a matrix ~$A$
\end{quote}
\vspace{-4mm}
to 
\vspace{-4mm}
\begin{quote} \em finding a \,$\la_*$
\,as a part of the least squares solution to 
~$\bdg(A,\la,X)=\bdo$
\end{quote}
\vspace{-4mm}
where ~$\bdg$ ~is the mapping defined in (\ref{jbitf}) with proper parameters.
~We shall show that the latter problem is a regularization of the 
former.

For any fixed matrix ~$A$, ~a local least squares solution 
~$(\hat\la,\hat{X})$ ~to the equation ~$\bdg(A,\la,X)=\bdo$ ~is the
minimum point to ~$\|\bdg(A,\la,X)\|_2$ ~in an open subset of 
~$\C\times\C^{n\times k}$ ~where
~$\bdg_{_{\la X}}(A,\hat\la,\hat{X})^\dagger\,\bdg(A,\hat\la,\hat{X}) 
\,=\, \bdo$ ~if ~$\bdg_{_{\la X}}(A,\hat\la,\hat{X})$ ~is injective.
~The least squares solution ~$(\la_*,X_*)$ ~of ~$\bdg(A,\la,X) = \bdo$
~can be solved by the Gauss-Newton iteration
\begin{equation}\label{gnit0}
(\la_{j+1},X_{j+1}) ~~=~~ (\la_j,X_j) - 
\bdg_{_{\la X}}(A,\la_j,X_j)^\dagger\,\bdg(A,\la_j,X_j), ~~~~j=0,1,\ldots
\end{equation}
based on the following local convergence lemma that is adapted 
from \cite[Lemma 2]{ZengAIF}.

\begin{lemma}\cite{ZengAIF}\label{l:gnit}
~Let ~$\bdg$ ~be the mapping in {\em (\ref{jbitf})}.
~For a fixed ~$A\in\C^{n\times n}$, ~assume ~$(\la_*,X_*)$ ~is a 
local least squares solution to ~$\bdg(A,\la,X)\,=\,\bdo$ ~with an
injective ~$\bdg_{_{\la X}}(A,\la_*,X_*)$.
~There is an open convex neighborhood ~$D$ ~of ~$(\la_*,X_*)$ ~and
constants ~$\zeta,\,\gamma\,>0$ ~such that
\begin{align}
\big\|\bdg_{_{\la X}}(A,\la,X)^\dagger\big\|_2 & ~~\le~~ \zeta,
\label{gnin1} \\
\big\|\bdg(A,\la,X)-\bdg(&A,\tilde\la,\tilde{X})-
 \bdg_{_{\la X}}(A,\tilde\la,\tilde{X})\,((\la,X)-
(\tilde\la,\tilde{X}))\big\|_2 \nonumber \\
& ~~\le~~
\gamma\,\big\|(\la,X)\!-\!(\tilde\la,\tilde{X})\big\|_2^2  \label{gnin2}
\end{align}
for all ~$(\la,X),(\tilde\la,\tilde{X})\,\in\,\overline{D}$. 
~Assume there is a ~$\sg<1$ ~such that,
for all ~$(\la,X)\,\in\,D$,
\begin{equation}\label{gnin3}
 \big\|\big(\bdg_{_{\la X}}(A,\la,X)^\dagger-
\bdg_{_{\la X}}(A,\la_*,X_*)^\dagger\big)\,\bdg(A,\la_*,X_*)
\big\|_2 ~~\le~~\sg\,\big\|(\la,X)-(\la_*,X_*)\big\|_2.
\end{equation}
Then, from all ~$(\la_0,X_0)\,\in\,D$ ~such that 
~$\big\|(\la_0,X_0)-(\la_*,X_*)\big\|_2 \,<\,\frac{1-\sg}{\zeta\,\gamma}$
~and 
\begin{equation}\label{gnin4} \big\{(\la,X)\in\C\times\C^{n\times k}~\big|~
\|(\la,X)-(\la_*,X_*)\|_2 \,<\,
\|(\la_0,X_0)-(\la_*,X_*)\|_2 \big\} ~~\subset~~ D,
\end{equation}
the Gauss-Newton iteration {\em (\ref{gnit0})}
is well defined in ~$D$, ~converges to ~$(\la_*,X_*)$ ~and satisfies
~$\big\|(\la_{j+1},X_{j+1})-(\la_*,X_*)\big\|_2 \,\le\, 
\mu\,\big\|(\la_j,X_j) -(\la_*,X_*)\big\|_2$ 
~for ~$j=0,1,\ldots$ 
~with ~$\mu = \sg+\zeta\,\gamma\,\big\|(\la_0,X_0)-(\la_*,X_*)\big\|_2
\,<\,1$.
\end{lemma}

When the matrix ~$A$ ~has an eigenvalue ~$\la_*$ ~of multiplicity support
~$m\times k$, ~there is an ~$X_*$ ~such that ~$(\la_*,X_*)$ ~is an exact
solution to ~$\bdg(A,\la,X)=\bdo$.
~However, when ~$A$ ~is known through its empirical data in ~$\tilde{A}$, 
~a local least squares solution ~$(\tilde\la,\tilde{X})$ ~to 
the equation ~$\bdg(\tilde{A},\la,X)=\bdo$ ~generally has a residual
~$\big\|\bdg(\tilde{A},\tilde\la,\tilde{X})\big\|_2>0$, ~and ~$\tilde\la$ ~may
not be an eigenvalue of either ~$A$ ~or ~$\tilde{A}$.
~For the convenience of elaboration, we call such a ~$\tilde\la$ ~an 
~$m\times k$ ~{\em pseudo-eigenvalue} of ~$\tilde{A}$.
~By changing the conventional problem of computing an eigenvalue to a modified
problem of finding
a pseudo-eigenvalue, the defective eigenproblem is regularized as a well-posed
problem as asserted in the main theorem of this paper.

\begin{theorem}[Pseudo-Eigenvalue Theorem]\label{t:pet}
~Let ~$\la_*$ ~be an eigenvalue of a matrix ~$A\in\C^{n\times n}$ ~with 
a multiplicity 
support ~$m\times k$ ~along with ~$X_*\in\C^{n\times k}$ ~satisfying
~$\bdg(A,\la_*,X_*)=\bdo$ ~where ~$\bdg$ ~is as in
{\em (\ref{jbitf})} with proper parameters ~$C$ ~and ~$S$.
~The following assertions hold.

\vspace{-4mm}
\begin{itemize}\parskip-0.5mm
\item[\em (i)] ~The exact eigenvalue ~$\la_*$ ~of ~$A$ ~is an ~$m\times k$ 
~pseudo-eigenvalue of ~$A$.
\item[\em (ii)] ~There are neighborhoods ~$\Phi$ ~of ~$A$ ~in 
~$\C^{n\times n}$ ~and ~$\Lambda$ ~of ~$\la_*$ ~in ~$\C$ ~such that every 
matrix ~$\tilde{A}\in\Phi$ ~has a unique ~$m\times k$ 
~pseudo-eigenvalue ~$\tilde\la\in\Lambda$ ~that is 
Lipschitz continuous with respect to ~$\tilde{A}$.
\item[\em (iii)] ~For every matrix 
~$\check{A}\,\in\,\Phi$ ~serving as empirical data of ~$A$, ~there is a unique
~$m\times k$ ~pseudo-eigenvalue ~$\check\la\in\Lambda$ ~of ~$\check{A}$ 
~such that
\begin{equation}\label{ferr}
  \big|\check\la-\la_*\big| ~~\le~~\tau_{A,m\times k}(\la_*)\, 
\big\|\check{A}-A\big\|_2
+O\big(\big\|\check{A}-A\big\|_2^2\big).
\end{equation}
\item[\em (iv)] ~The ~$\check\la$ ~in {\em (iii)} is an exact 
eigenvalue of ~$\check{A}+E\,\check{X}^\dagger$ ~with a Jordan
block of size at least ~$k$
~where ~$\check{X}$ ~is the least squares solution of ~$\bdg(\check{A},
\check\la,X) = \bdo$ ~and 
~$E = (\check{A}-\check\la\,I)\,\check{X}-\check{X}\, S$.
~When ~$\check{X}^\h\check{X}=I$, 
~the backward error ~$\|E\,\check{X}^\dagger\|_{_F}$ ~is bounded by
~$\big\|\bdg(\check{A},\check\la,\check{X})\big\|_2$.
\end{itemize}
\end{theorem}

\begin{proof}
~The assertion (i) ~is a result of Lemma~\ref{l:jbit} (i).
~For any ~$r>0$, ~denote 
~$\Psi_r\,=\,\big\{(\la,X)\in\C\times\C^{n\times k} ~\big|~
\|(\la,X)-(\la_*,X_*)\|_2<r \big\}$ ~and let ~$r_0>0$ ~such that
~$\{A\}\times\overline{\Psi_{r_0}}$
~is a subset of ~$\Sigma$ ~in Lemma~\ref{l:jbith}.
~Let ~$r\,\in\,(0,r_0)$. 
~Assume there is a matrix ~$\tilde{A}$ ~with 
~$\|\tilde{A}-A\|_2 < \eps$ ~for any ~$\eps>0$ ~such that 
~$\min_{(\la,X)\in\overline{\Psi_r}}\big\|\bdg(\tilde{A},\la,X)\big\|_2$
~is not attainable in ~$\Psi_r$.
~Let ~$\eps\rightarrow 0$.
~Then ~$\tilde{A}\rightarrow A$ ~and there exists an 
~$(\hat\la,\hat{X})\in\overline{\Psi_r}\setminus\Psi_r$ ~such that
~$\|\bdg(A,\hat\la,\hat{X})\|_2$ ~is the minimum 0 of 
~$\big\|\bdg(A,\la,X)\big\|_2$ ~for ~$(\la,X)\in\overline{\Psi_r}$
~and ~$(\hat\la,\hat{X})\,\ne\,(\la_*,X_*)$.
~This is a contradiction to Lemma~\ref{l:jbith}. 
~As a result, there is a neighborhood ~$\Phi_r$ ~of ~$A$ ~for every 
~$r\in (0,r_0)$ ~such that ~$\min_{(\la,X)\in\Psi_r}\,\|\bdg(\tilde{A},\la,X)
\|_2$ ~is attainable at 
certain ~$(\tilde\la,\tilde{X})\,\in\,\Psi_r$ ~for every ~$\tilde{A}\in\Phi_r$, 
~implying the existence of the pseudo-eigenvalue ~$\tilde\la$.

By Lemma~\ref{l:gnit}, we can assume ~$r_1\in (0,r_0)$ ~is small so that 
the inequalities (\ref{gnin1}), (\ref{gnin2}) and
(\ref{gnin3}) hold for ~$\sg=0$ ~and
~$\|(\la,X)-(\tilde\la,\tilde{X})\|_2 \,<\, \frac{1}{2\,(2\zeta)\,(2\gamma)}$ 
~for all ~$(\la,X),\,(\tilde\la,\tilde{X})\,\in\,\overline{\Psi_{r_1}}$.
~By the continuity of ~$\bdg$, ~the corresponding ~$\Phi_{r_1}$ ~can be
chosen so that, for every ~$\hat{A}\,\in\,\Phi_{r_1}$
~with a local minimum point ~$(\hat\la,\hat{X})\,\in\,\Psi_{r_1}$ ~for
~$\|\bdg(\hat{A},\la,X)\|_2$,
~we have ~$\big\|\bdg_{_{\la X}}(\hat{A},\la,X)^\dagger\big\|_2 \,<\, 
2\zeta$,
\begin{eqnarray*}
\mbox{\small $
\big\|\bdg(\hat{A},\la,X)-\bdg(\hat{A},\tilde\la,\tilde{X})-
\bdg_{_{\la X}}(\hat{A},\tilde\la,\tilde{X})\,((\la,X)-
(\tilde\la,\tilde{X}))\big\|_2$}
&~<~& 2\gamma\,\big\|(\la,X)-(\tilde\la,\tilde{X})\big\|_2^2, \\
\big\|\big(\bdg_{_{\la X}}(\hat{A},\la,X)^\dagger-
\bdg_{_{\la X}}(\hat{A},\hat{\la},\hat{X})^\dagger\big)\,
\bdg(\hat{A},\hat\la,\hat{X})
\big\|_2 &~\le~&\mbox{\footnotesize $\frac{1}{2}$}\,
\big\|(\la,X)-(\hat{\la},\hat{X})\big\|_2
\end{eqnarray*}
for all ~$(\la,X),\,(\tilde\la,\tilde{X})\,\in\,\Psi_{r_1}$.
~Let ~$r_2=\frac{1}{3}\,r_1$, ~$\Psi=\Psi_{r_2}$ ~and
~$\Phi=\Phi_{r_1}\cap\Phi_{r_2}$.
~For every ~$\hat{A}\in\Phi$, ~the minimum of ~$\|\bdg(\hat{A},\la,X)\|_2$
~is attainable at ~$(\hat\la,\hat{X})\in\Psi$ ~and, for any initial iterate
~$(\la_0,X_0)\in\Psi$, ~we have
~$\|(\la_0,X_0)-(\hat\la,\hat{X})\|_2 \,<\, 
\frac{1}{2\,(2\zeta)\,(2\gamma)}
\,=\,\frac{1-\frac{1}{2}}{(2\zeta)\,(2\gamma)}$
~and the set ~$\Omega \,=\,\big\{ (\la,X)\in\C\times\C^{n\times k} \,\big|\, 
\|(\la,X)-(\hat\la,\hat{X})\|_2 \,<\,\|(\la_0,X_0)-(\hat\la,\hat{X})\|_2
\big\}$
~is in ~$\Psi_{r_1}$
~since, for every ~$(\la,X)\in\Omega$, ~we have
\begin{eqnarray*} \|(\la,X)-(\la_*,X_*)\|_2 &~~\le~~&
\|(\la,X)-(\hat\la,\hat{X})\|_2 +\|(\hat\la,\hat{X})-(\la_*,X_*)\|_2 \\
& < & \|(\la_0,X_0)-(\hat\la,\hat{X})\|_2 
+r_2 \\
& \le & \|(\la_0,X_0)-(\la_*,X_*)\|_2 +
\|(\la_*,X_*)-(\hat\la,\hat{X})\|_2 + r_2 \\
&~~<~~& r_2+r_2+r_2 ~~=~~ r_1
\end{eqnarray*}
By Lemma~\ref{l:gnit}, for every ~$(\la_0,X_0)\,\in\,\Psi$, ~the Gauss-Newton 
iteration on the equation ~$\bdg(\hat{A},\la,X)=\bdo$ 
~converges to ~$(\hat\la,\hat{X})$.
~This local minimum point ~$(\hat\la,\hat{X})$ ~is unique in ~$\Psi$ ~because,
assuming there is another minimum point ~$(\check\la,\check{X})\in\Psi$ ~of
~$\|\bdg(\hat{A},\la,X)\|_2$, ~the Gauss-Newton iteration converges 
to ~$(\hat\la,\hat{X})$ ~from the initial point ~$(\check\la,\check{X})$.
~On the other hand, the Gauss-Newton iteration from the local minimum point
~$(\check\la,\check{X})$ ~must stay at ~$(\check\la,\check{X})$, ~implying
~$(\check\la,\check{X})\,=\,(\hat\la,\hat{X})$.

On the Lipschitz continuity of the pseudo-eigenvalue, let ~$\tilde{A},
\check{A}\in\Phi$ ~with minimum points ~$(\tilde\la,\tilde{X})$
~and 
~$(\check\la,\check{X})$ ~of ~$\|\bdg(\tilde{A},\la,X)\|_2$ 
~and ~$\|\bdg(\check{A},\la,X)\|_2$ ~respectively in ~$\Psi$.
~The one-step Gauss-Newton iterate 
~$(\la_1,X_1) \,=\, (\tilde\la,\tilde{X}) - 
\bdg_{_{\la X}}(\check{A},\tilde\la,\tilde{X})^\dagger\,
\bdg(\check{A},\tilde\la,\tilde{X})$
~from ~$(\tilde\la,\tilde{X})$ ~on the equation
~$\bdg(\check{A},\la,X)=\bdo$ ~toward ~$(\check\la,\check{X})$ 
~yields the inequality 
~$\big\|(\la_1,X_1)-(\check\la,\check{X})\big\|_2 \,\le\,\mu\,
\big\| (\tilde\la,\tilde{X}) - (\check\la,\check{X}) \big\|_2 $ 
~with ~$0\le\mu<1$ ~by Lemma~\ref{l:gnit}.
~Thus
\begin{eqnarray*}
 \big\|(\check\la,\check{X})-(\tilde\la,\tilde{X})\big\|_2 &~~\le~~& 
 \big\|(\check\la,\check{X})-(\la_1,X_1)\big\|_2 + 
 \big\| (\la_1,X_1) - (\tilde\la,\tilde{X}) \big\|_2  \\
 & \le &
\mu\,\big\|(\check\la,\check{X})-(\tilde\la,\tilde{X})\big\|_2  + 
 \big\| (\la_1,X_1) - (\tilde\la,\tilde{X}) \big\|_2 
\end{eqnarray*}
Using the identity 
~$\bdg_{_{\la X}}(\tilde{A},\tilde\la,\tilde{X})^\dagger\,
\bdg(\tilde{A},\tilde\la,\tilde{X}) \,=\,\bdo$ ~and the Lipschitz 
continuity of ~$\bdg$ ~and ~$\bdg_{_{\la X}}$, 
~there is a constant ~$\gamma$ ~such that
\begin{eqnarray*}
\big\|(\check\la,\check{X})-(\tilde\la,\tilde{X})\big\|_2
&\le&
\frac{1}{1\!-\!\mu}\,\big\|(\la_1,X_1)-(\tilde\la,\tilde{X})\big\|_2 \\
& = & \frac{1}{1\!-\!\mu}\,\big\|
\bdg_{_{\la X}} (\check{A},\tilde\la,\tilde{X})^\dagger 
\bdg(\check{A},\tilde\la,\tilde{X}) - 
\bdg_{_{\la X}} (\tilde{A},\tilde\la,\tilde{X})^\dagger 
\bdg(\tilde{A},\tilde\la,\tilde{X})
\big\|_2 \\
& \le & \frac{1}{1\!-\!\mu}\,
\Big(\big\|\bdg_{_{\la X}}(\check{A},\tilde\la,\tilde{X})^\dagger\big\|_2\,
\big\|\bdg(\check{A},\tilde\la,\tilde{X})-
\bdg(\tilde{A},\tilde\la,\tilde{X})\big\|_2\\
& & ~~~~~~~~~~+
\big\|\bdg_{_{\la X}} (\check{A},\tilde\la,\tilde{X})^\dagger-
\bdg_{_{\la X}}(\tilde{A},\tilde\la,\tilde{X})^\dagger\big\|_2 
\big\|\bdg(\tilde{A},\tilde\la,\tilde{X})\big\|_2\Big)\\
&~\le~&  \gamma\,\big\|\tilde{A}-\check{A}\big\|_2
\end{eqnarray*}
for all ~$\tilde{A},\check{A}\in\Phi$.
~Namely, the ~$m\times k$ ~pseudo-eigenvalue is Lipschitz continuous with 
respect to the matrix.
~By setting ~$(\tilde{A},\tilde\la,\tilde{X}) = (A,\la_*,X_*)$
~in the above inequalities we have (\ref{ferr}) because 
~$\big\|\bdg(A,\la_*,X_*)\big\|_2=0$.
~Thus ~$\mu=0$ ~and (iii) is proved.

For the assertion (iv), ~the ~$\check{X}$ ~is of full rank
since ~$X_*$ ~is and the least squares solution of
~$\bdg(G,\la,X)=\bdo$ ~is continuous, implying 
~$\check{X}^\dagger\,\check{X} = I$ ~and thus
~$E = E\,\check{X}^\dagger\,\check{X}$, ~leading to 
~$(\check{A}-E\,\check{X}^\dagger-\check\la I)\check{X} \,=\,\check{X}\,S$.
~The eigenvalue ~$\check\la$ ~of ~$\check{A}+E\,\check{X}^\dagger$ 
~corresponds to a Jordan block of size at least ~$k$
~since ~$S$ ~in (\ref{S}) is nilpotent of rank ~$k-1$.
~\QED
\end{proof}

The Pseudo-Eigenvalue Theorem establishes a rigorous and thorough
regularization of the ill-posed problem in computing a defective eigenvalue 
so that the problem of computing a pseudo-eigenvalue enjoys unique existence 
and Lipschitz continuity of the solution that approximates the underlying 
defective eigenvalue with an error bound proportional to the data error,
reaffirming the ~$m\times k$ ~condition number
as a bona fide sensitivity measure of an eigenvalue whether it is defective
or not.
~This regularization makes it possible to compute defective eigenvalues
accurately using floating point arithmetic even if the matrix is
perturbed. 

\vspace{-4mm}
\section{An algorithm for computing a defective eigenvalue}

\vspace{-4mm}
Theorem~\ref{t:pet} sets the foundation for accurate computation 
of a defective eigenvalue.
~We assume the given matrix ~$A$ ~is 
the data representation of an underlying matrix possessing a defective 
eigenvalue,
an initial estimate ~$\la_0$ ~is close to that eigenvalue,
and the multiplicity support ~$m\times k$ ~is known, identified or 
estimated (more to that later in \S\ref{s:id}). 
~By Lemma~\ref{l:jbit}, the proper parameter ~$C$ ~is in an open dense 
subset of 
~$\C^{n\times m}$ ~so that we can set ~$C$ ~at random.
~With ~$C$ ~available, we can then set up
\begin{eqnarray}
\bdx_1^{(0)} & ~~=~~& 
\mbox{\footnotesize 
$\left[\begin{array}{c} A-\la_0\,I \\ C^\h 
\end{array}\right]^\dagger\,\left[\begin{array}{c} \bdo \\ T_{1:m,1}
\end{array}\right]$} \label{x10} \\
\bdx_{j+1}^{(0)} &~~=~~&  \al_{j}\,\mbox{\footnotesize 
$\left[\begin{array}{c} A-\la_0\,I \\ C^\h 
\end{array}\right]^\dagger\,\left[\begin{array}{c} \bdx_j^{(0)} \\ \bdo
\end{array}\right]$} 
~~~\mbox{for}~~ j = 1,\ldots,k-1
\label{xj1alj1} \\
S & = & \mbox{\scriptsize 
$\left[\begin{array}{c|ccc} 0 &  ~\al_1 & & \\ \vdots & & \ddots & \\ 
0 & & & \al_{k-1}\\ \hline
0 & ~0 & \cdots & 0 \\ \end{array}\right]$} \label{xj1a1j2}
\end{eqnarray}
where, for ~$j=1,\ldots,k-1$, ~the scalar ~$\al_j$ ~scales ~$\bdx_{j+1}^{(0)}$
~to a unit vector.
~Denote ~$X_0 = \blb \bdx_1^{(0)},\cdots,\bdx_k^{(0)}\brb$.
~Then ~$\bdg(A,\la_0,X_0) \,\approx\,\bdo$ ~and we apply the Gauss-Newton 
iteration (\ref{gnit0})
that converges to ~$(\la_*,\,X_*)$ ~assuming the initial estimate ~$\la_0$
~is sufficiently close to ~$\la_*$.
~When the iteration stops at the ~$j$-th step, a QR decomposition of
the matrix representing ~$\bdg_{_{\la X}}(A,\la_j,X_j)$ ~is available 
and thus an estimate ~$\big\|\bdg_{_{\la X}}(A,\la_j,X_j)^\dagger\big\|_2$
~of the ~$m\times k$ ~condition number can be computed by a couple of steps
of inverse iteration \cite{li-zeng-03} with a negligible cost.
~A pseudo-code of Algorithm {\sc PseudoEig} is given in 
Fig.~\ref{f:jbit}.

\begin{figure}[ht]
\begin{center}
\shabox{\hspace{0mm}\parbox{4.5in}{
\vspace{-1mm}
\index{Algorithm PseudoEig}
{\bf Algorithm} {\sc PseudoEig}
\begin{itemize}
\item[] \hspace{-9mm}{\sc Input}: matrix ~$A$, 
~eigenvalue estimate ~$\la_0$,
~multiplicity support ~$m$, ~$k$
\begin{itemize}
\item set ~$C$ ~as a random ~$n\times m$ ~matrix and ~$\bdx_1^{(0)}$ 
~as in (\ref{x10})
\item set ~$\bdx_{2}^{(0)},\cdots,\bdx_{k}^{(0)}$ ~by (\ref{xj1alj1})
\item set ~$X_0 = \blb \bdx_{1}^{(0)},\cdots,\bdx_{k}^{(0)}\brb$,
~$S$ ~as in (\ref{xj1a1j2}) and ~$\bdg$ ~as in (\ref{jbitf})
\item for ~$j = 0, 1, \ldots$ ~do 
\begin{itemize}
\item solve 
~$\displaystyle \bdg_{_{\la X}}(A,\la_j,X_j) \,(\sg, Y) ~=~ \bdg(A,\la_j,X_j)$ 
~for the least squares solution ~$(\sg,Y)$
\item set ~$\la_{j+1} = \la_j-\sg$, ~$X_{j+1} = X_j-Y$.
\item if ~$\big\|\bdg(A,\la_j,X_j)\big\|_2 < \big\|\bdg(A,\la_{j+1},X_{j+1})
\big\|_2$  ~then
\begin{itemize}
\item[] set ~$(\hat{\la},\hat{X}) = (\la_j,X_j)$, ~break the loop. ~end if
\end{itemize}
\end{itemize}
\item[] end do
\end{itemize}
\item[] \hspace{-9mm}
{\sc Output}: pseudo-eigenvalue ~$\hat{\la}$, ~backward error bound
~$\big\|\bdg(A,\hat\la,\hat{X})\big\|_2\,\big\|X^\dagger\big\|_2$, 
~$m\times k$ ~condition number 
~$\big\|\bdg_{_{\la X}}(A,\hat\la,\hat{X})^\dagger\big\|_2$
\end{itemize}
}}
\end{center}
\caption{Algorithm {\sc PseudoEig}} \label{f:jbit}
\end{figure}
\normalsize

\vspace{-4mm}
\section{Taking advantage of the Jacobian structure}

\vspace{-4mm}
The matrix ~$\bdg_{_{\la X}}(A,\la_j,X_j)$
is pleasantly structured with a proper 
arrangement so that the cost of its QR decomposition can be reduced 
substantially.  
~Let ~$X=\blb \bdx_1,\cdots,\bdx_k \brb$, ~the image 
~$\bdg(A,\la,X)$ ~can be arranged as 
\[
\mbox{\tiny $\left[\begin{array}{rrrrrl}
C^\h\,\bdx_k  &&&&& - T_{1:m,k}\\
(A -\la I)\,\bdx_k & - s_{k-1,k}\,\bdx_{k-1} & - s_{k-2,k}\,\bdx_{k-2} &
-\cdots & -s_{1k}\,\bdx_1 &   \\
& C^\h\,\bdx_{k-1}  &&&& - T_{1:m,k-1}\\
& (A -\la I)\,\bdx_{k-1} & - s_{k-2,k-1}\,\bdx_{k-2} & 
-\cdots & -s_{1,k-1}\,\bdx_1 &   \\
& & \ddots~~~~~~~~~~~~~~~~~~ & \ddots~~ & \vdots~~~~~~ & ~~~~~~~~\vdots \\
 & & \ddots~~~~~~~~ & &\ddots~~~~~~\vdots~~~~~~ & ~~~~~~~~\vdots \\
 & & \ddots & &-s_{12}\,\bdx_1 & ~~~~~~~~\vdots \\
 & & & \ddots & C^\h\,\bdx_1 &  -T_{1:m,1} \\
 & & & & (A-\la\,I)\,\bdx_1 &   \\
\end{array}\right]$}.
\]
As a result, the partial Jacobian matrix in a blockwise upper-triangular form%
\footnote{Matlab code is available at 
{\tt homepages.neiu.edu/$\sim$zzeng/pseudoeig.html}.
}.
\begin{eqnarray*}
\lefteqn{\mbox{\scriptsize $\pd{\bdg(A,\la,X)}{(\bdx_k,\ldots,\bdx_1,\la)}$}
~~=}\\&& \mbox{\tiny $\left[\begin{array}{cccccc}  
\multicolumn{1}{|c}{
\begin{array}{c} C^\h \\ A-\la I \end{array}} & 
\begin{array}{c} O \\ -s_{k-1,k}\,I \end{array} & 
\begin{array}{c} O \\ -s_{k-2,k}\,I \end{array} & 
\begin{array}{c} \cdots \\ \cdots \end{array} & 
\begin{array}{c} O \\ -s_{1k}\,I \end{array}  &
\begin{array}{c} \bdo \\ -\bdx_k \end{array}  \\ \cline{1-1}
&
\multicolumn{1}{|c}{
\begin{array}{c} C^\h \\ A-\la I \end{array}} & 
\begin{array}{c} O \\ -s_{k-2,k-1}\,I \end{array} & 
\begin{array}{c} \cdots \\ \cdots \end{array} & 
\begin{array}{c} O \\ -s_{1,k-1}\,I \end{array}  &
\begin{array}{c} \bdo \\ -\bdx_{k-1} \end{array}  \\ \cline{2-2}
& ~~~~~\ddots & ~~~~~~~~\ddots & & \vdots & \vdots \\
& & \ddots~~~ & \ddots  & \vdots & \vdots \\
& & & 
\multicolumn{1}{|c}{
\begin{array}{c} C^\h \\ A-\la I \end{array}}  & 
\begin{array}{c} O \\ -s_{12}\,I \end{array} & 
\begin{array}{c} \bdo \\ -\bdx_{2} \end{array}  \\ \cline{4-4}
& & &  & 
\multicolumn{1}{|c}{
\begin{array}{c} C^\h \\ A-\la I \end{array}}  & 
\begin{array}{c} \bdo \\ -\bdx_{1} \end{array} \\ \cline{5-6}
\end{array}\right]$}.
\end{eqnarray*}
We can further assume the matrix ~$A$ ~is already reduced to a Hessenberg form
or even Schur form. 
~Then 
\[ \left[\begin{array}{c} C^\h \\ A-\la I \end{array}\right] ~~=~~
\mbox{\tiny $\left[\begin{array}{cccc} * & * & \cdots & * \\
\vdots & \vdots & \ddots & \vdots \\ * & ~*~ & ~\cdots~ & * \\
  & * & \cdots & * \\ &        & \ddots & \vdots \\
  &        &        &  * \end{array}\right]$}
\]
is nearly upper-triangular with ~$m+1$ ~subdiagonal lines of nonzero entries.
~The QR decomposition of the partial Jacobian 
~$\bdg_{\bdx_k\cdots\bdx_1\,\la}(A,\la,X)$ ~can then be carried out by 
a sequence of standard textbook Householder transformations.

The main cost of Algorithm {\sc PseudoEig} occurs at solving 
the linear equation
\[ \bdg_{_{\la X}}(A,\la_j,X_j)(\sg, Y) ~~=~~ \bdg(A,\la_j,X_j)
\]
for the least squares solution
~$(\sg,Y)$ ~in ~$\C\times\C^{n\times k}$. 
~The structure of the partial Jacobian matrix ~$\bdg_{_{\la X}}(A,\la_j,X_j)$ 
~can be taken advantage of if the QR decomposition is needed.
~It is also suitable to apply an iterative method for large sparse matrices
particularly if ~$A$ ~is sparse.

\vspace{-4mm}
\section{Identifying the multiplicity support}\label{s:id}

\vspace{-4mm}
The geometric multiplicity can be identified with 
numerical rank-revealing.
~Let ~$\la_0$ ~be an initial estimate of ~$\la_*\in\eig(A)$ ~in 
Lemma~\ref{l:jbit} and assume
\[ |\la_0-\la_*| ~~<~~ \theta ~~<~~ 
\min_{\la\in\eig(A)\setminus\{\la_*\}}\,\big|\la-\la_0\big|.
\]
The geometric multiplicity of ~$\la_*$ ~can be computed as the numerical
nullity of ~$A-\la_0 I$ ~within the error tolerance ~$\theta$ ~defined as
\begin{equation}\label{gmest}  m ~~=~~ 
\max\big\{j ~\big|~ \sg_{n-j+1}(A-\la_0 I) < \theta\}
\end{equation}
where ~$\sg_i(\cdot)$ ~is the \,$i$-th largest singular value of ~$(\cdot)$.
~A misidentification of the geometric multiplicity can be detected.
~Underestimating ~$m$ ~results in an undersized ~$C$ ~in (\ref{jbitf}) so that
both 
~{\scriptsize $\left[\begin{array}{c} A-\la_* I \\ C^\h \end{array}\right]$}
~and the partial Jacobian ~$\bdg_{_{\la X}}(A,\la_*,X_*)$ ~are rank-deficient.
~Overestimating ~$m$ ~renders the system {\scriptsize
~$\left[\begin{array}{c} A-\la_* I \\ C^\h \end{array}\right]\,
\bdu \,=\,
 \left[\begin{array}{c} \bdo \\ T_{1:m,1} \end{array}\right]$}
~inconsistent with a large residual norm.
~During an iteration in which ~$(\la_j,X_j)$ ~approaches ~$(\la_*,X_*)$,
a large condition number of the partial Jacobian ~$\bdg_{_{\la X}}(A,\la_j,X_j)$
~indicates a likely underestimated geometric multiplicity and a large
residual ~$\big\|\bdg(A,\la_j,X_j)\big\|_2$ ~suggests a possible 
overestimation.

If the geometric multiplicity is identified, it is possible to find 
the Segre anchor by a searching scheme based on the condition
number of the Jacobian ~$\bdg_{_{\la X}}(A,\la_j,X_j)$ ~as shown in the
following example.

\begin{example} \em
~Let ~$A$ ~be the matrix 
\[\mbox{\tiny $\left[\begin{array}{rrrrrrrrrrrrrrrrrrrr}
0 \!& 4 \!& 0 \!&-4 \!& 0 \!&-2 \!& 1 \!& 0 \!& 0 \!&-1 \!&-1 \!&-1 \!&-1 \!& 2 \!& 1 \!& 0 \!& 0 \!&-1 \!& 0 \!& 0\\
0 \!& 3 \!& 3 \!&-4 \!& 1 \!& 0 \!& 4 \!& 1 \!&-1 \!&-1 \!& 0 \!& 0 \!& 0 \!& 0 \!& 0 \!&-2 \!& 0 \!& 0 \!&-1 \!&-1\\
1 \!&-4 \!& 2 \!&10 \!&-3 \!& 1 \!&-7 \!&-2 \!&-1 \!& 3 \!& 2 \!& 0 \!& 0 \!&-1 \!& 0 \!& 3 \!&-1 \!& 1 \!& 1 \!& 2\\
-1 \!&-1 \!& 2 \!& 5 \!&-2 \!&-1 \!&-5 \!&-1 \!&-1 \!& 2 \!& 0 \!&-1 \!& 0 \!& 1 \!& 0 \!& 2 \!&-1 \!&-1 \!&-1 \!& 1\\
-1 \!&-2 \!& 2 \!& 1 \!& 1 \!& 1 \!&-1 \!& 0 \!&-2 \!& 1 \!& 0 \!& 0 \!& 0 \!& 1 \!& 0 \!& 0 \!&-1 \!&-1 \!& 0 \!& 0\\
1 \!& 4 \!& 1 \!&-12 \!& 4 \!& 2 \!&13 \!& 3 \!& 0 \!&-4 \!& 0 \!& 0 \!&-2 \!&-1 \!& 1 \!&-6 \!& 1 \!& 1 \!& 0 \!&-3\\
-1 \!&-1 \!& 1 \!& 5 \!&-2 \!& 0 \!&-4 \!&-2 \!& 0 \!& 1 \!&-1 \!& 0 \!& 0 \!& 1 \!& 0 \!& 4 \!& 0 \!&-1 \!&-1 \!& 2\\
1 \!& 2 \!&-4 \!& 0 \!& 1 \!& 0 \!& 1 \!& 1 \!& 4 \!&-2 \!&-1 \!& 1 \!& 0 \!&-1 \!& 0 \!& 1 \!& 2 \!& 1 \!& 0 \!& 1\\
0 \!&-5 \!& 2 \!&10 \!&-5 \!&-1 \!& -10 \!&-2 \!&-1 \!& 6 \!& 3 \!&-2 \!& 0 \!& 0 \!& 0 \!& 3 \!&-3 \!& 0 \!& 1 \!& 2\\
 1 \!& 1 \!& 1 \!&-1 \!& 2 \!& 2 \!& 4 \!& 0 \!& 1 \!&-1 \!&-1 \!& 2 \!& 0 \!&-1 \!& 0 \!& 0 \!& 2 \!& 1 \!&-1 \!& 0\\
 1 \!&-1 \!& 0 \!& 2 \!& 1 \!& 2 \!& 1 \!& 0 \!& 1 \!&-1 \!& 3 \!& 2 \!& 0 \!&-1 \!& 0 \!& 0 \!& 1 \!& 1 \!&-1 \!& 0\\
-1 \!&-3 \!& 0 \!& 5 \!&-1 \!& 2 \!&-4 \!&-1 \!& 0 \!& 1 \!&-1 \!& 4 \!& 4 \!& 1 \!&-2 \!& 2 \!& 0 \!&-1 \!& 0 \!& 1\\
-2 \!& 0 \!& 0 \!&-1 \!& 0 \!& 0 \!& 0 \!& 0 \!& 0 \!& 0 \!&-1 \!& 0 \!& 3 \!& 2 \!& 0 \!& 0 \!& 0 \!&-1 \!& 0 \!& 0\\
-3 \!& 4 \!&-1 \!&-4 \!& 0 \!&-2 \!& 1 \!& 0 \!& 0 \!&-1 \!&-1 \!&-1 \!&-2 \!& 5 \!& 2 \!& 0 \!& 0 \!&-1 \!& 1 \!& 0\\
-2 \!& 0 \!& 0 \!&-1 \!& 0 \!& 0 \!& 0 \!& 0 \!& 0 \!& 0 \!&-1 \!& 0 \!& 0 \!& 2 \!& 3 \!& 0 \!& 0 \!&-1 \!& 0 \!& 0\\
 0 \!& 0 \!&-2 \!& 3 \!& 1 \!& 2 \!&-1 \!&-1 \!& 2 \!&-2 \!&-2 \!& 2 \!& 0 \!& 0 \!& 0 \!& 5 \!& 2 \!& 0 \!& 0 \!& 1\\
 6 \!& 3 \!&-6 \!& 3 \!& 6 \!& 4 \!& 7 \!& 0 \!& 7 \!&-7 \!& 1 \!& 5 \!&-2 \!&-6 \!& 1 \!& 0 \!& 8 \!& 6 \!&-1 \!& 0\\
 0 \!& 2 \!&-4 \!&-4 \!& 1 \!&-1 \!& 4 \!& 1 \!& 0 \!&-1 \!& 0 \!&-1 \!&-1 \!& 0 \!& 1 \!&-2 \!& 0 \!& 3 \!& 4 \!&-1\\
 1 \!&-4 \!&-1 \!&11 \!&-4 \!& 1 \!&-8 \!&-3 \!&-1 \!& 3 \!& 2 \!& 0 \!& 0 \!&-1 \!& 0 \!& 4 \!&-1 \!& 1 \!& 4 \!& 3\\
 0 \!& 0 \!&-1 \!& 1 \!&-2 \!& 0 \!&-1 \!&-2 \!& 1 \!& 0 \!& 0 \!& 0 \!& 0 \!& 0 \!& 0 \!& 1 \!& 0 \!& 0 \!& 0 \!& 4
\end{array}\right]$}
\]
where ~$\eig(A) = \{2,3\}$ ~of nonzeero Segre characteristics 
~$\{4,3,3\}$ ~and ~$\{5,5\}$ ~respectively.
~Applying the Francis QR algorithm implemented in Matlab yields computed 
eigenvalues scattered around ~$\la_1=2.0$ ~and ~$\la_2=3.0$:

\tiny
\[ \begin{array}{rcr}
2.000118556521482 + 0.000118397929590i
&~~~~~~~& 3.000398490901253 + 0.001224915665189i \\
2.000118556521482 - 0.000118397929590i
&& 
3.000398490901253 - 0.001224915665189i\\ 
1.999881443477439 + 0.000118714860725i
&& 
3.000646066870935 + 0.000469627646058i\\ 
1.999881443477439 - 0.000118714860725i
&& 
3.000646066870935 - 0.000469627646058i\\ 
2.000013778528383 + 0.000018105742295i
&& 
3.001287762162967 + 0.000000000000000i\\ 
2.000013778528383 - 0.000018105742295i
&& 
2.999753002133234 + 0.000759191914332i\\ 
2.000008786021464 + 0.000020979720849i
&& 
2.999753002133234 - 0.000759191914332i\\ 
2.000008786021464 - 0.000020979720849i
&& 
2.998957628017279 + 0.000757681758834i\\ 
1.999977435451235 + 0.000002873978888i
&& 
2.998957628017279 - 0.000757681758834i\\ 
1.999977435451235 - 0.000002873978888i
&& 
2.999201861991639 + 0.000000000000000i 
\end{array}
\] \normalsize
Using two computed eigenvalues above, say 
\[\tilde\la_0 ~~=~~ 
\mbox{\tiny 1.999881443477439 - 0.000118714860725i} ~~~~\mbox{and}~~~
\hat\la_0 ~~=~~ \mbox{\tiny 3.001287762162967 + 0.000000000000000i}
\]
as initial estimates of the defective eigenvalues, smallest singular values
~of ~$A-\tilde\la_0 I$ ~and ~$A-\hat\la_0 I$ ~can be computed using a
rank-revealing method as
\tiny
\[ \begin{array}{rcr}
\sg_j(A-\tilde\la_0 I):~~~~~~~~~~ & & 
\sg_j(A-\hat\la_0 I):~~~~~~~~~~ \\
\cdots & & \cdots \\
   0.084065699924186 &~~~~~~~~~~~~~&    0.070046163725993\\
   0.049368630759014 &~~~~~&    0.054661269198836\\
   \mathbf{0.000000000003635} &~~~~~&    0.036234932328447\\
   \mathbf{0.000000000000280} &~~~~~&    \mathbf{0.000000000000001}\\
   \mathbf{0.000000000000001} &~~~~~&    \mathbf{0.000000000000000}\\
\end{array}
\]\normalsize
indicating the geometric multiplicities ~$3$ ~and ~$2$ ~respectively.

Set the geometric multiplicities for the initial eigenvalue estimate
~$\tilde\la_0$ ~and ~$\hat\la_0$ ~as ~$3$ ~and ~$2$ ~respectively.
~Applying Algorithm~{\sc PseudoEig} with increasing input 
~$k=1,2,\ldots,$ ~as estimated Segre anchors, 
~we list the computed eigenvalues, ~$m\times k$
~condition numbers and residual norms in Table~\ref{t:sca}.
~At ~$\la_1$, ~for instance, underestimated values ~$k=1,2$
~render the ~$m\times k$ ~condition numbers as large as ~$10^8$ ~and 
the residuals to be tiny, while the overestimated value ~$k=4$ ~leads to
a drastic increase of residual from ~$10^{-16}$ ~to ~$10^{-3}$ ~but maintains
the moderate ~$m\times k$ ~condition number, as shown in Table~\ref{t:sca}.
~Similar effect of increasing estimated values of the Segre
anchor at ~$\la_2$ ~can be observed consistently. \QED
\end{example} 

\begin{table}[htb]
\begin{center} \small
\begin{tabular}{||r||c|r|l||} \hline\hline
test & \multicolumn{3}{c||}{at ~$\la_1 = 2$,
~Segre anchor ~$k=3$} \\ \cline{2-4}
$k$ ~value  & computed eigenvalue & condition number & residual norm \\ \hline
$k=1$ & \tiny\tt  1.999881443477439 - 0.000118714860725i%
& \tiny\tt 560995239.6 & \tiny\tt 0.000000000000001 \\
$k=2$ & \tiny\tt   1.999999993438010 - 0.000000011324234i%
& \tiny\tt 147603979.2 & \tiny\tt 0.000000000000001 \\
$\rightarrow~k=3$ & \tiny\tt\tt  2.000000000000000 - 0.000000000000000i%
& \tiny\tt  58.7 & {\tiny\tt  0.0000000000000006} $\leftarrow$ \\
$k=4$ & \tiny\tt  2.109885640097783 - 0.004348977611146i%
& \tiny\tt 24.1 & \tiny\tt 0.007 \\ \hline \hline
test & \multicolumn{3}{c||}{at ~$\la_2 = 3$, 
~Segre anchor ~$k=5$}\\ \cline{2-4}
$k$ ~value & computed eigenvalue & condition number & residual norm \\ \hline
$k=1$ & \tiny\tt  3.001287762162967 & \tiny\tt 2161090332264.6  & 
\tiny\tt 0.000000000000003 \\ 
$k=2$ & \tiny\tt 3.001287762162967 & \tiny\tt 7962600062.8  & 
\tiny\tt 0.0000000000005 \\ 
$k=3$ & \tiny\tt 3.001287762162967 & \tiny\tt 4556940.4  & 
\tiny\tt 0.000000003 \\ 
$k=4$ & \tiny\tt 3.000000013572103 & \tiny\tt 687859583.9  & 
{\tiny\tt 0.0000000000000007} \\ 
$\rightarrow~k=5$ & \tiny\tt 3.000000000000000 & \tiny\tt 33.9  & 
{\tiny\tt 0.0000000000000007} ~$\leftarrow$\\ 
$k=6$ & \tiny\tt 3.002451613695432 & \tiny\tt 34.1  & 
\tiny\tt 0.007 \\ \hline\hline
\end{tabular}
\end{center}
\caption{\small Effect of increasing estimated Segre anchors:
~Underestimated values yield large condition numbers of the Jacobian and
overestimated values lead to large residual norms.
~The results using the correct anchors are pointed out with arrows.}
\label{t:sca} 
\end{table}

Identify multiplicity support in practical computation 
can be challenging.
~It is certainly a subject that is worth further studies.

\vspace{-4mm}
\section{Improving accuracy with orthonormalization}\label{s:orth}

\vspace{-4mm}
Algorithm {\sc PseudoEig} uses a simple nilpotent matrix ~$S$ ~with only
one superdiagonal line of nonzero entries.
~By Lemma~\ref{l:jbit} (iii), we can modify ~$C$ ~and ~$S$ ~as parameters
of ~$\bdg$ ~so that the matrix component ~$\check{X}$ ~of the solution to
~$\bdg(A,\la_*,\check{X})=\bdo$ ~has orthonormal columns.
~The orthonormalization can be carried out 
by the following process:

\begin{itemize}\parskip0mm
\item[$-$] 
Execute Algorithm {\sc PseudoEig} and obtain output 
~$\hat\la, \hat{X}, C, S$.
\item[$-$] Normalize ~$\hat{X}_{1:n,1}$ ~and adjust ~$s_{12}$ ~so
that ~$(A-\hat\la I)\,\hat{X} \approx \hat{X}\,S$ ~still holds.
\item[$-$] Reset ~$C_{1:n,1}$ ~as ~$\hat{X}_{1:n,1}$.
\item[$-$] Reset
~$\hat{X}_{1:n,2:k}$ ~as ~$\hat{X}_{1:n,2:k} - \hat{X}_{1:n,1}\, 
(\hat{X}_{1:n,1})^\h \hat{X}_{1:n,2:k}$.
\item[$-$] Reset
~$S_{1,1:k}$ ~as 
~$S_{1,1:k}+(\hat{X}_{1:n,1})^\h \hat{X}_{1:n,2:k}\,S_{2:k,1:k}$.
\item[$-$] Obtain the thin QR decomposition ~$\hat{X}=Q\,R$.
\item[$-$] Reset ~$S$ ~as ~$R\,S\,R^{-1}$ ~in the mapping ~$\bdg$.
\item[$-$] Set the initial iterate ~$(\la_0,X_0) = (\hat\la,Q)$ ~for
the Gauss-Newton iteration (\ref{gnit0}).
\end{itemize}

The advantage of such an orthonormalization is intuitively clear.
~When we solve for the least squares solution ~$(\tilde{\la},\tilde{X})$ 
~of the equation ~$\bdg(A,\la,X)=\bdo$ ~minimizing the magnitude of the residual 
~$(A-\tilde\la\,I)\tilde{X} - \tilde{X}\,S = E$, ~the backward error 
given in Theorem~\ref{t:pet} (iv) is 
~$\|E\|_2\,\big\|\tilde{X}^\dagger\big\|_2$. 
~When the norm ~$\|\tilde{X}^\dagger\|_2$ ~is large, ~minimizing the residual
norm ~$\|E\|_2$ ~may not achieve the highest attainable backward accuracy.
~If the columns of ~$\tilde{X}$ ~are orthonormal, however, the norm 
~$\|\tilde{X}^\dagger\|_2=1$ ~and the least squares solution that minimizing
the residual norm ~$\|E\|_2$ ~directly minimizes the backward error bound.

\begin{example}\em
~Consider the matrix
\begin{equation}\label{jbiteA}  A ~~=~~ \mbox{\tiny $\left[\begin{array}{rrrrr}
2 & 1 & & & \\ & -8 & 1 & & \\ & & 2 & 1 & \\ & & & 2 & 1 \\ 
& -10000 & 1000 & -100 & 12
\end{array}\right]$}
\end{equation}
with an exact eigenvalue ~$\la_*=2$ ~and the multiplicity support ~$1\times 5$.
~A straightforward application of Algorithm~{\sc PseudoEig} in Matlab
yields
\begin{eqnarray*}
\tilde\la &~~=~~ & \mbox{\scriptsize 1.999999999999748} \\
S & = &  \mbox{\tiny $\left[\begin{array}{ccccc}
  0 &  0.100686223197184 &  0  &   0  &       0 \\
  0 &  0 &  0.680272615629152  &   0  &       0 \\
  0 &  0 & 0 & 0.786924421181882  &     0 \\
  0 &  0 & 0 & 0 & 0.922632632948520 \\
  0 &  0 & 0 & 0 &                 0
\end{array}\right]$} \\
\tilde{X}  & = & 
\mbox{\tiny $\left[\begin{array}{rrrrr}
   1.00502786434024 & 0.10210319200724 & .07627239342106 & 
.06542640851275 & .06584192219606 \\
  -0.00000000000025 & 0.10119245986833 & .06945800549083 & 
.06002060904501 & .06036453955047 \\
  -0.00000000000253 & 1.01192459868319 & .76341851426484 & 
.65486429121738 & .65902236805904 \\
  -0.00000000000000 & -0.00000000000051 & .68838459356550 & 
.60075267245722 & .60419916522969 \\
  -0.00000000000000 & -0.00000000000000 & -.00000000000052 & 
.54170664784191 & .55427401993992
\end{array}\right]$}
\end{eqnarray*}
The residual norm
\[  \big\|(A-\tilde\la I)\,\tilde{X}-\tilde{X}\,S\big\|_F ~~\approx~~ 
4.5\times 10^{-14}
\]
can not be minimized further with the unit round-off about ~$10^{-16}$
~considering ~$\|A\|_2 \approx 10^4$.
~The backward error
\[  \big\|(A-\tilde\la I)\,\tilde{X}-\tilde{X}\,S
\big\|_F\,\big\|\tilde{X}^\dagger\big\|_2  ~~\approx~~ 1.3\times 10^{-9}
\]
is not small enough.
~After orthonormalization and resetting the resulting parameter ~$C$ ~and ~$S$
~in ~$\bdg$ ~in (\ref{jbitf}), we apply the Gauss-Newton iteration again
and obtain
\begin{eqnarray*}
\hat\la & ~~=~~ & \mbox{\scriptsize $2.000000000000000$} \\
S & = & \mbox{\tiny $\left[\begin{array}{rrrrr}
0 & 0.09950371902 & -0.00990049999 & 0.00099000050 & -0.99498744208 \\
0 & 0 & 1.00493781395 & -0.00098508732 & 0.99004950866 \\
0 & 0 & 0 & 1.00004900870 & -0.09850873917 \\
0 & 0 & 0 & 0 & 10050.38307728113 \\
0 & 0 & 0 & 0   &          0
\end{array}\right]$} \\
\hat{X} & = &
\mbox{\tiny $\left[\begin{array}{rrrrr}
  -1.0 & -0.000000000002531 &  -0.000000000000000 & -0.000000000000000 
& -0.000000000000000 \\
   0.0 &  -0.099503719021067 & 0.009900499987341 & -0.000990000499944 
& 0.994987442082474 \\
   0.0 & -0.995037190210673 & -0.000990049999192 & 0.000099000049994 
& -0.099498744209908 \\
   0.0 & -0.000000000000000 & -0.999950498725976 & -0.000009900005712 
& 0.009949874421338 \\
   0.0 & -0.000000000000000 &-0.000000000000000 &-0.999999505000536 
&-0.000994987010945
\end{array}\right]$}
\end{eqnarray*}
The residual practically stays about the same magnitude
\[  \big\|(A-\hat\la I)\,\hat{X}-\hat{X}\,S\big\|_F ~~\approx~~ 
1.25\times 10^{-14}
\]
but the backward error improves substantially to 
\[  \big\|(A-\hat\la I)\,\hat{X}-\hat{X}\,S\|_F\,
\big\|\hat{X}^\dagger\big\|_2  ~~\approx~~ 1.25\times 10^{-14}
\]
as ~$\|\hat{X}^\dagger\|_2 \approx 1$.
~More importantly, the forward accuracy of the computed eigenvalue improves 
by 3 additional accurate digits. \QED
\end{example}

When the given matrix represents perturbed data, the orthonormalization 
seems to be more significant in improving the accuracy, as shown in the example
below.

\begin{example}\em
~Using a random perturbation of magnitude about ~$10^{-5}$, ~let 
\begin{equation}\label{AA105}
 \tilde{A} ~~=~~ A + 10^{-5}\,\mbox{\tiny $\left[\begin{array}{rrrrr}
  -0.092 & -0.653  &  -0.201 &  -0.416 &  -0.787 \\
  -0.135 &  -0.218 &   0.054 &  -0.136 &  -0.255 \\
   0.651 &   0.663 &  -0.166 &  -0.969 &  -0.603 \\
  -0.833 &   0.607 &   0.314 &   0.969 &  -0.020 \\
  -0.733 &  -0.879 &   0.256 &  -0.665 &  -0.321 
\end{array}\right]$}
\end{equation}
be the data representation of the matrix ~$A$ ~in (\ref{jbiteA}). 
~Table~\ref{t:baor} lists the computed eigenvalues, residual norms, backward 
errors and forward errors before and after orthonormalization.
~The results show a substantial improvement on the both forward
and backward errors even though the residual magnitudes roughly stay the same.
\QED

\begin{table}[ht]
\begin{center}
\begin{tabular}{|c||c|c|} \hline
& ~before orthonormalization~ & ~after orthonormalization~ \\ \hline
computed eigenvalue &  {\bf 2.00}4413315474177 & {\bf 2.000000}343999377  \\
residual norm &  $2.3\times 10^{-6}$  &  $2.9\times 10^{-6}$  \\
backward error & $ 6.7\times 10^{-2} $ &  $2.9\times 10^{-6}$  \\
forward error & $4.4\times 10^{-3}$ & $3.4\times 10^{-7}$
\\  \hline
\end{tabular}
\end{center}
\caption{Comparison between computing results with or without orthonormalization
of the ~$X$ ~component of the least squares solution to 
~$\bdg(\tilde{A},\la,X)=\bdo$ ~for the matrix ~$\tilde{A}$ ~in 
(\ref{AA105}) at the eigenvalue ~$\la=2$.
~Correct digits of computed eigenvalues are highlighted in boldface.}
\label{t:baor}
\end{table}
\end{example}

\vspace{-4mm}
\section{What kind of eigenvalues are ill-conditioned, and in what sense?}
\label{s:w}

\vspace{-4mm}
The well documented claim that a defective eigenvalue is infinitely sensitive 
to perturbations requires an oft-missing clarification: ~Its unbounded 
sensitivity is with respect to {\em arbitrary} perturbations.
~The sensitivity of a defective eigenvalue is finitely bounded by 
the spectral projector norm divided by the multiplicity if the perturbation
is constrained to maintain the multiplicity, or by the multiplicity support
condition number if the multiplicity support remains unchanged.

Furthermore, the above sensitivity assertions and clarifications are 
applicable on the problem of finding eigenvalues in its strictly narrow sense.
~In the sense of computing a multiple eigenvalue via a cluster mean
provided that the cluster can be grouped correctly, the sensitivity is still
bounded by spectral projector norm divided by the multiplicity.
~The problem of finding a defective eigenvalue in the sense of computing a 
pseudo-eigenvalue elaborated in this paper also enjoys a finitely bounded 
sensitivity in terms of the multiplicity support condition number.

Of course, the problem can still be ill-conditioned even if the sensitivity 
is finitely bounded.
~In the following example, the matrix ~$A$ ~has an eigenvalue of multiplicity
7 and the spectral projector norm is large, so the eigenvalue is 
ill-conditioned in this sense.
~On the other hand, the same eigenvalue is well-conditioned in 
multiplicity support sensitivity. 
~Interestingly, this is not a contradiction at all.
~The conflicting sensitivity measures imply that the cluster mean is not
accurate for approximating the eigenvalue but the pseudo-eigenvalue is, and 
Algorithm {\sc PsedoEig} converges to the defective eigenvalue with all 
the digits correct.

\begin{example}\label{e:4}\em
~A simple eigenvalue 
~$\la_1 = 2.001$ ~and a defective eigenvalue ~$\la_2=2$ ~with the
Segre characteristic ~$\{5,2,0,\ldots \}$, i.e. multiplicity support 
~$2\times 2$, ~exist for
\[  A~~=~~ \mbox{\tiny $\left[\begin{array}{cccccccc}
3.006 &   2  &     1.005 &  -1.001 &  -0.002 &   -0.001 &   -0.001 &  -1\\
 5     &   2  &     5     &  -1     &  -2     &   -1     &   -1     &   0\\
-5.006 &  -3  &    -3.005 &   2.001 &   3.002 &    2.001 &    0.001 &   2\\
-6     &  -1  &    -6     &   3     &   5     &    3     &    0     &   1\\
-5     &  -1  &    -5     &   1     &   6     &    3     &    0     &   1\\
 1     &   0  &     1     &   0     &  -1     &    1     &    0     &   0\\
-4     &  -2  &    -4     &   1     &   3     &    2     &    2     &   2\\
 5     &   0  &     5     &  -1     &  -2     &   -1     &   -1     &   2\\
\end{array}\right].$}
\]
Let ~$P_2$ ~be the spectral projector associated with ~$\la_2=2$.
~The defective eigenvalue ~$\la_2$ ~is {\em both} highly ill-conditioned 
in spectral projector norm and almost perfectly conditioned measured by its 
~$2\times 2$ ~condition number with a sharp contrast:
\[ \mbox{$\frac{1}{m}$}\|P_2\|_2 \,\approx\,4.05\times 10^{14} ~~~~\mbox{while}~~~~
\tau_{A,2\times 2}(\la_2) \le 19.95.
\]
This may seem to be a contradiction except it is not. 
~Both conditions accurately measure the sensitivities of same end (finding
the defective eigenvalue) through different means (cluster mean versas 
pseudo-eigenvalue).
~The Francis QR algorithm implemented in Matlab produces computed eigenvalues

\tiny
\begin{verbatim}
            2.003667055821394,                       2.001912473859015 + 0.002992156370408i,
            1.996674198110247,                       2.001912473859015 - 0.002992156370408i,
            2.000000046670435,                       1.998416899175164 + 0.002994143122392i,
            1.999999953329568,                       1.998416899175164 - 0.002994143122392i.
\end{verbatim}\normalsize 

There is no apparent way to group 7 computed eigenvalues to use the cluster 
mean for the defective eigenvalue even if we know the multiplicity is 7.
~Out of all 8 possible groups of 7 eigenvalues, the best approximation
to ~$\la_2=2.0$ ~by the average is ~$2.000142850475652$ ~with a substantial 
error ~$1.4\times 10^{-4}$ ~predicted by the spectral projector norm.
~In contrast, Algorithm {\sc PseudoEig} accurately converges to
~$\la_2=2.0$ ~with an error below the unit round off ~$2.2\times 10^{-16}$ 
~using the correct multiplicity support ~$2\times 2$ ~that can easily be
identified using the method in \S\ref{s:id}, as accurately 
predicted by the ~$2\times 2$ ~condition number. 

This seemingly contradicting sensitivities can be explained by the fact that
there are infinitely many matrices nearby possessing a single eigenvalue of 
nonzero Segre characteristic ~$\{6,2\}$ ~within 2-norm distances of
~$5.2\times 10^{-5}$. 
~Namely, such a small perturbation increases the multiplicity from 7 to 8 but
can not increase the multiplicity support ~$2\times 2$.
~Using the publicly available Matlab functionality 
{\sc NumericalJordanForm} on the matrix ~$A$ ~with error tolerance ~$10^{-5}$
~in the software package {\sc NAClab}%
\footnote{http://homepages.neiu.edu/$\sim$naclab}
for numerical algebraic computation, we obtain approximately nearest matrix ~$B$ 
~with a single eigenvalue associated with Jordan blocks sizes 6 and 2 
with first 14 digits of its entries given as

\tiny\begin{verbatim}

   3.0059955942896  1.9999978851470  1.0049959180573 -1.0010020728471 -0.0020046893569 -0.0010002300301 -0.0010132897111 -0.9999977586058
   4.9999998736434  1.9999937661529  5.0000001193777 -1.0000000065301 -2.0000000129428 -0.9999999934070 -0.9999999926252 -0.0000169379637
  -5.0060008381014 -3.0000021146845 -3.0050076499797  2.0009979267360  3.0019953102172  2.0009997699688  0.0009867094117  2.0000022421372
  -5.9999927405774 -1.0000021677309 -6.0000074892946  2.9999962015789  4.9999997701478  2.9999999999775 -0.0000002324249  0.9999978331627
  -4.9999995930006 -0.9999961877596 -5.0000095377366  0.9999880178349  5.9999870716625  3.0000002295144 -0.0000150335883  0.9999994536709
   0.9999971940837 -0.0000010574036  1.0000006545259 -0.0000047736356 -1.0000023807994  0.9999966612522 -0.0000036987224 -0.0000054161918
  -4.0000092166543 -1.9999997569827 -3.9999765043908  1.0000142841290  3.0000142782479  2.0000000005386  2.0000249853630  2.0000002431865
   4.9999998983338  0.0000026655939  5.0000001062663 -0.9999999958672 -1.9999999894366 -1.0000000065916 -1.0000000120790  2.0000133696815

\end{verbatim}\normalsize

The spectrum of ~$B$ ~consists of a single eigenvalue ~$\la=2.00125$. 
~This lurking nearby matrix indicates that the multiplicity 7 of 
~$\la_2=2.0\in\eig(A)$ ~can be increased to 8 ~with a small perturbation
~$\|A-B\|_2$, ~which is exactly the kind of cases where spectral projectors
have large norms as elaborated by Kahan \cite{Kahan72}
and grouping method fails. 
~However, those nearby defective matrices have the same multiplicity support
~$2\times 2$, ~implying a small perturbation does not increase either
the geometric multiplicity or the Segre anchor.
~As a result, the multiplicity support condition number is benign, and 
computing the defective eigenvalue via pseudo-eigenvalue is stable.

Interestingly, even though the matrix ~$B$ ~is only known via the above 
empirical data, the spectral projector associated with its eigenvalue 
~$2.00125$ ~is known to be  identity since there is only one distinct eigenvalue. 
~Consequently, the mean of all approximate eigenvalues computed by Francis 
QR algorithm is  2.000124999999987 ~with 14 digits accuracy, same as the
empirical data.
~Algorithm {\sc PseudoEig} produces the ~$2\times 2$ ~pseudo-eigenvalue 
2.000125000000078 with the same number of correct digits due to a small 
~$2\times 2$ ~condition number 14.47.
~The software {\sc NumericalJordanForm} accurately produces the Jordan Canonical
Forms of both matrices ~$A$ ~and ~$B$.
~\QED

\end{example}

\bibliographystyle{siamplain}

\end{document}